\documentclass[11pt,preprint,twoside]{elsarticle}
\usepackage[a4paper,top=2.5cm,bottom=2.5cm,left=2.5cm,right=2.5cm]{geometry}
\usepackage[utf8]{inputenc}

\usepackage{amsmath}

\usepackage{amssymb}

\usepackage{amsthm}

\usepackage{arydshln}

\usepackage{charter}

\usepackage{colortbl}

\usepackage{fancyhdr}

\usepackage{enumerate}

\usepackage{enumitem}

\usepackage{float}

\usepackage{hyperref}

\usepackage{mathtools}

\usepackage{tabularx}

\usepackage{tikz}

    \usetikzlibrary{decorations.pathmorphing,decorations.markings,cd}

\usepackage{ifthen}

\newtheorem{thm}{Theorem}[section]
\newtheorem{prop}[thm]{Proposition}
\newtheorem{cor}{Corollary}[thm]
\newtheorem{lemma}[thm]{Lemma}

\theoremstyle{definition}

\newtheorem{defn}[thm]{Definition}
\newtheorem{eg}[thm]{Example}
\newtheorem{remark}[thm]{Remark}

\newcommand{\Mod}[1]{{\text{Mod-}{#1}}}
\newcommand{\cat}[1]{\mathcal{#1}}

\newcommand{\der}[1]{{\mathcal{D}} \left({#1} \right) }
\newcommand{\derstar}[1]{{\mathcal{D}}^{*} \! \left({#1} \right)}
\newcommand{\derstarnone}[1]{{\mathcal{D}}^{*} \! \left( \text{Mod} \right)}

\newcommand{\homotopy}[1]{{\mathcal{K}}\left({#1} \right)}

\newcommand{\im}[1]{\text{im} ( {#1} )}
\newcommand{\kernel}[1]{\text{ker} ( {#1} )}

\newcommand{\tens}[3]{{#1} {\otimes_{#2}} {#3}}
\newcommand{\ltens}[3]{{#1} \otimes_{#2}^L {#3}}

\newcommand{\Homgroup}[3]{\text{Hom}_{#1} \left( {#2}, {#3} \right)}
\newcommand{\Hom}[5]{\text{Hom}_{#1} \left( {_{#2}{#3}},{_{#4}{#5}} \right)}

\newcommand{\rHom}[3]{\text{RHom}_{#1} \left( {#2},{#3} \right)}
\newcommand{\dHom}[3]{\text{Hom}_{\der{#1}} \left( {#2},{#3} \right)}
\newcommand{\kHom}[3]{\text{Hom}_{\homotopy{#1}} \left( {#2},{#3} \right)}

\newcommand{\ext}[4]{\text{Ext}^{#1}_{#2} ( {#3}, {#4} )}
\newcommand{\tor}[4]{\text{Tor}^{#1}_{#2} ( {#3}, {#4} )}

\newcommand{\Inj}[1]{\text{Inj-}{#1}}
\newcommand{\lsub}[2]{\text{Loc}_{#1} \left( {#2} \right)}
\newcommand{\lsubinj}[1]{\lsub{#1}{\Inj{#1}}}

\newcommand{\Proj}[1]{\text{Proj-}{#1}}
\newcommand{\cosub}[2]{\text{Coloc}_{#1} \left( {#2} \right)}
\newcommand{\cosubproj}[1]{\cosub{#1}{\Proj{#1}}}

\newcommand{\ds}[3]{\bigoplus_{{#1} \in {#2}}{#3}}

\newcommand{\tri}[6]{{#1} \xrightarrow{#4} {#2} \xrightarrow{#5} {#3} \xrightarrow{#6} {#1}[1]}
\newcommand{\ses}[5]{0 \rightarrow {#1} \xrightarrow{#4} {#2} \xrightarrow{#5} {#3} \rightarrow 0}

\newcommand{\dual}[1]{{D#1}}

\newcommand{\res}[3]{\ifthenelse{\equal{#3}{}}{\text{Res}_{#1}^{#2}}{\text{Res}_{#1}^{#2}{\left( {#3} \right)}}}

\newcommand{\ind}[3]{\ifthenelse{\equal{#3}{}}{\text{Ind}_{#1}^{#2}}{\text{Ind}_{#1}^{#2}{\left( {#3} \right)}}}

\newcommand{\lind}[3]{\ifthenelse{\equal{#3}{}}{\text{\textbf{L}Ind}_{#1}^{#2}}{\text{\textbf{L}Ind}_{#1}^{#2}{\left( {#3} \right)}}}
    
\newcommand{\coind}[3]{\ifthenelse{\equal{#3}{}}{\text{Coind}_{#1}^{#2}}{\text{Coind}_{#1}^{#2}{\left( {#3} \right)}}}

\newcommand{\rcoind}[3]{\ifthenelse{\equal{#3}{}}{\text{\textbf{R}Coind}_{#1}^{#2}}{\text{\textbf{R}Coind}_{#1}^{#2}{\left( {#3} \right)}}}

\newcommand{\cogenerator}[1]{{#1}^{\ast}}

\newcommand{\triext}[2]{{#1} \ltimes {#2}}

\newcommand{\twobytwo}[4]{\left( \begin{matrix} {#1} & {#2} \\ {#3} & {#4} \end{matrix} \right)}
\newcommand{\integers}[0]{\mathbb{Z}}
\newcommand{\rationals}[0]{\mathbb{Q}}

\newcommand{\projdim}[2]{\text{proj.dim}_{#1} \left({#2} \right)}
\newcommand{\injdim}[2]{\text{inj.dim}_{#1} \left({#2} \right)}

\newcommand{\relprojdim}[2]{\text{pgl}_{#1} \left({#2} \right)}
\newcommand{\relinjdim}[2]{\text{igl}_{#1} \left({#2} \right)}

\newcommand{\quotient}[3]{{#1}/{#2}^{#3}}

\newcommand{\function}[3]{#1 \colon #2 \rightarrow #3}

\newcommand{\goodtrunabove}[2]{\tau_{\leq {#1}}({#2})}

\newcommand{\cohomology}[2]{H^{#1}\left( {#2} \right) }

\newcommand{\supremum}[1]{\text{sup}\{ #1 \}}

\newcommand{\cx}[3]{\left( {#1}^{#3}, {#2}^{#3} \right)_{{#3} \in \integers} } 

\newcommand{\rad}[1]{\text{rad} \left( {#1} \right)}

\makeatletter
\def\ps@pprintTitle{ 
 \let\@oddhead\@empty
 \let\@evenhead\@empty
 \def\@oddfoot{\centerline{\thepage}}
 \let\@evenfoot\@oddfoot}
\makeatother

\usepackage{fancyhdr}
\begin{document}

\begin{frontmatter}

\title{\textbf{Ring Constructions and Generation of the Unbounded Derived Module Category}}

\author{Charley Cummings}

\begin{abstract}
We consider the smallest triangulated subcategory of the unbounded derived module category of a ring that contains the injective modules and is closed under set indexed coproducts. If this subcategory is the entire derived category, then we say that injectives generate for the ring. In particular, we ask whether, if injectives generate for a collection of rings, do injectives generate for related ring constructions and vice versa. We provide sufficient conditions for this statement to hold for various constructions including recollements, ring extensions and module category equivalences.
\end{abstract}

\begin{keyword}
Derived categories \sep Homological conjectures \sep Recollements
\end{keyword}

\end{frontmatter}

\setcounter{tocdepth}{1}
\tableofcontents

\vspace{0.25cm}
\section{Introduction}

For several decades the derived category has been used as a tool to understand problems in representation theory. One longstanding problem that has fuelled a lot of research in this area is the finitistic dimension conjecture. Huisgen-Zimmermann gives a summary on the history of this conjecture in \cite{Huisgen1995}. Recently, Rickard \cite{Rickard2019} showed that a finite dimensional algebra satisfies the finitistic dimension conjecture if its derived module category is generated, as a triangulated category with coproducts, by the injective modules. In this paper we provide techniques to detect if this generation property holds for rings obtained from various ring constructions.

In a talk in 2001, Keller \cite{Keller2001} considered the smallest triangulated subcategory of the unbounded derived module category of a ring that contains the injective modules and is closed under set indexed coproducts. If this subcategory is the entire unbounded derived module category then we say that \textit{injectives generate} for the ring. In particular, Keller noted that if injectives generate for an algebra then the algebra satisfies some open conjectures in representation theory, including the Nunke condition and the Generalised Nakayama conjecture. Happel gives an overview of these conjectures in \cite{Happel1991}. Keller asked whether `injectives generate' also implied the finitistic dimension conjecture, to which Rickard \cite[Theorem 4.3]{Rickard2019} provided an affirmative answer.

It is not the case that injectives generate for all rings; one counterexample is the polynomial ring in infinitely many variables \cite[Theorem 3.5]{Rickard2019}. However, there is no known example of a finite dimensional algebra over a field for which injectives do not generate. There are many classes of finite dimensional algebras for which injectives do generate, including commutative algebras, Gorenstein algebras and monomial algebras \cite[Theorem 8.1]{Rickard2019}. Note that if injectives generate for a class of finite dimensional algebras then the finitistic dimension conjecture also holds for these algebras.

There is extensive work into finding methods that can be used to identify rings that satisfy particular homological properties. One such approach is to consider a collection of related rings and ask whether, if some of the rings have a property, do the others as well. For example, this method has been applied to the finitistic dimension conjecture with recollements of derived categories \cite{Happel1993, Chen2017}, ring homomorphisms \cite{Xi2004,Xi2006,Xi2008,Wang2017}, and operations defined on the quiver of a quiver algebra \cite{Fuller1992, Green2018, Bravo2020}. We apply the same philosophy to the property `injectives generate'. This leads us to employ reduction techniques originally used in calculating the finitistic dimension of a ring to check if injectives generate for a ring, including the arrow removal for quiver algebras defined by Green, Psaroudakis and Solberg \cite[Section 4]{Green2018}. 

One of the most general ring constructions is given by two rings $A$ and $B$ and a ring homomorphism $\function{f}{B}{A}$ between them. We provide sufficient conditions on a ring homomorphism such that if injectives generate for $A$ then injectives generate for $B$ and vice versa. The conditions we supply are satisfied by many familiar ring constructions, including those shown in Theorem~\ref{thm:introringhom}.

\begin{thm} \label{thm:introringhom}
    Let $A$ and $B$ be rings. Suppose that injectives generate for $B$. If one of the following holds then injectives generate for $A$.
    \begin{enumerate}[label = (\roman*)]
        \item $A$ is the trivial extension $\triext{B}{M}{}$ for a $(B,B)$-bimodule $M$ such that the ideal $(0,M) \leq A$ has finite flat dimension as a left $A$-module. [Lemma~\ref{lemma:nilpotentidealig}, Example~\ref{eg:trivialextension}]
        \item $A$ is a Frobenius extension of $B$. [Lemma~\ref{lem:ring_ext_relatively_proj_inj_generation}, Example~\ref{eg:relativelyinjectivefrobeniusandalmostexcellentextensions}]
        \item $A$ is an almost excellent extension of $B$. [Lemma~\ref{lem:ring_ext_relatively_proj_inj_generation}, Example~\ref{eg:relativelyinjectivefrobeniusandalmostexcellentextensions}]
    \end{enumerate}
\end{thm}

In particular, consider the triangular matrix ring of two rings $B$ and $C$ and a $(C,B)$-bimodule $M$, denoted by
\begin{equation*}
    A := \twobytwo{C}{_CM_B}{0}{B}.
\end{equation*}
The ring $A$ is isomorphic to the trivial extension $\triext{\left(C \times B\right)}{M}{}$. Hence, if the nilpotent ideal
\begin{equation*}
    \twobytwo{0}{_CM_B}{0}{0} \leq A,
\end{equation*}
has finite flat dimension as a left $A$-module and injectives generate for $C \times B$, then Theorem~\ref{thm:introringhom} applies and injectives generate for $A$.

The triangular matrix ring also induces a recollement of derived module categories. A recollement is a diagram of six functors between three derived module categories emulating a short exact sequence of rings introduced by Be\u{\i}linson, Bernstein and Deligne \cite{Beuilinson1982}. We consider recollements of unbounded derived module categories that restrict to recollements of bounded (above or below) derived module categories. This requires all six functors to restrict to functors of bounded (respectively above or below) derived categories. Angeleri H\"{u}gel, Koenig, Liu and Yang provide necessary and sufficient conditions for a recollement to restrict to a bounded (above) recollement \cite[Proposition 4.8 and Proposition 4.11]{Huegel2017}. In Proposition~\ref{prop:recollementbdedbelowconditions} we prove an analogous result to characterise when a recollement restricts to a bounded below recollement. These characterisations can be used to prove the following theorem.

\begin{thm} \label{thm:introrecollement}
    Let $(R) = (\der{B}, \der{A}, \der{C})$ be a recollement of unbounded derived module categories with $A$ a finite dimensional algebra over a field.
    Suppose that injectives generate for both $B$ and $C$. If one of the following conditions holds, then injectives generate for $A$.
        \begin{enumerate}[label = (\roman*)]
            \item The recollement $(R)$ restricts to a recollement of bounded below derived categories. [Proposition~\ref{prop:recollementbdedbelow}]
            \item The recollement $(R)$ restricts to a recollement of bounded above derived categories. [Proposition~\ref{prop:recollementsbdedabove}]
        \end{enumerate}
\end{thm}
The recollement induced by a triangular matrix algebra restricts to a recollement of bounded above derived categories so Theorem~\ref{thm:introrecollement} can be applied to triangular matrix algebras. Note that the techniques used in the proof of Theorem~\ref{thm:introrecollement} can also be used to prove similar results about recollements of dg algebras. In particular, we apply these techniques to the `vertex removal' operation defined by Green, Psaroudakis and Solberg \cite[Section 5]{Green2018} with respect to injective generation.

\subsection*{Layout of the paper}

The paper starts in Section~\ref{section:localisingsubcategory} by recalling definitions and useful properties of localising and colocalising subcategories. Section~\ref{section:tensorproductalgebra} provides a straightforward example of the techniques used to prove `injectives generate' statements by considering the tensor product algebra over a field. In Section~\ref{section:equivalences} we show that separable equivalence preserves the property `injectives generate'. Section~\ref{section:ringhomomorphisms} considers general ring homomorphisms and includes the proof of Theorem~\ref{thm:introringhom}. The paper concludes with various `injectives generate' results for recollements of derived categories in Section~\ref{section:recollements} including Theorem~\ref{thm:introrecollement} and recollements of module categories in Section~\ref{section:dgrecollements}.

\subsection*{Acknowledgement}

I would like to thank my Ph.D. supervisor Jeremy Rickard for his guidance and many useful discussions.

\section{Preliminaries} \label{section:localisingsubcategory}

In this section we provide the required definitions and preliminary results that will be used throughout the paper. Subsection~\ref{subsection:(co)localisingsubcategories} recalls the definition of localising and colocalising subcategories of the derived module category. In Subsection~\ref{subsection:adjointfunctors} we focus on showing when triangle functors of the derived category preserve specific properties of complexes.

\subsection{Notation} \label{subsection:notation}

Firstly, we fix some notation. All rings and ring homomorphisms are unital and modules are right modules unless otherwise stated. For a ring $A$ and left $A$-module $M$ we denote $\cogenerator{M}$ to be the right $A$-module $\Hom{\mathbb{Z}}{}{_AM}{}{\mathbb{Q}/\mathbb{Z}}$. 

We use the following notation for various categories.
\begin{itemize}
    \item $\Mod{A}$ is the category of all $A$-modules.
    \item $\Inj{A}$ is the category of all injective $A$-modules and $\Proj{A}$ is the category of all projective $A$-modules.
    \item $\homotopy{A}$ is the unbounded homotopy category of cochain complexes of $A$-modules.
    \item $\der{A}$ is the unbounded derived category of cochain complexes of $A$-modules. $\derstar{A}$ with $\ast \in \{-, +, b\}$ is the bounded above, bounded below and bounded derived category respectively.
\end{itemize}

We shall be considering functors between derived categories and properties that they preserve. For brevity, we shall often abuse terminology by writing, for example, that a functor preserves bounded complexes of projectives when we mean that it preserves the property of being quasi-isomorphic to a bounded complex of projectives.

Let $A$ and $B$ be rings and $\function{f}{B}{A}$ be a ring homomorphism. The functors induced by $f$ will be denoted as follows.
\begin{itemize}
    \item Induction, $\ind{B}{A}{} \coloneqq \function{\tens{-}{B}{A_A}}{\Mod{B}}{\Mod{A}}$,
    \item Restriction, $\res{B}{A}{} \coloneqq \function{\Hom{A}{B}{A}{}{-}}{\Mod{A}}{\Mod{B}}$,
    \item Coinduction, $\coind{B}{A}{} \coloneqq \function{\Hom{B}{}{_AA}{}{-}}{\Mod{B}}{\Mod{A}}$.
\end{itemize}

\subsection{(Co)Localising subcategories} \label{subsection:(co)localisingsubcategories}

There are many ways to generate the unbounded derived category of a ring, here we focus on generation via localising and colocalising subcategories. First we recall their definitions.

\begin{defn}[(Co)Localising Subcategory]
    Let $A$ be a ring and $\cat{S}$ be a class of complexes in $\der{A}$. 
    \begin{itemize}
        \item A localising subcategory of $\der{A}$ is a triangulated subcategory of $\der{A}$ closed under set indexed coproducts. The smallest localising subcategory of $\der{A}$ containing $\cat{S}$ is denoted by $\lsub{A}{\cat{S}}$.
        \item A colocalising subcategory of $\der{A}$ is a triangulated subcategory of $\der{A}$ closed under set indexed products. The smallest colocalising subcategory of $\der{A}$ containing $\cat{S}$ is denoted by $\cosub{A}{\cat{S}}$.
    \end{itemize}
\end{defn}

There are some well known properties of localising and colocalising subcategories which can be found in
\cite[Proposition 2.1]{Rickard2019}.

\begin{lemma}\cite[Proposition 2.1]{Rickard2019}
    Let $A$ be a ring and $\cat{C}$ be a triangulated subcategory of $\der{A}$.
    \begin{enumerate}[label = (\roman*)]
        \item If $\cat{C}$ is either a localising subcategory or a colocalising subcategory of $\der{A}$ then $\cat{C}$ is closed under direct summands.
        \item Let $X = \cx{X}{d}{i} \in \der{A}$ be bounded. If the module $X^i$ is in $\cat{C}$ for all $i \in \integers$, then $X$ is in $\cat{C}$.
    \end{enumerate}
\end{lemma}

Throughout this paper we investigate when a localising subcategory or colocalising subcategory of $\der{A}$ generated by some class of complexes $\cat{S}$ is the entire unbounded derived module category. 

\begin{defn} \label{defn:sgeneratesda}
    Let $A$ be a ring and $\cat{S}$ be a class of complexes in $\der{A}$.
    \begin{itemize}
        \item If $\lsub{A}{\cat{S}} = \der{A}$ then we say that $\cat{S}$ generates $\der{A}$.
        \item If $\cosub{A}{\cat{S}} = \der{A}$ then we say that $\cat{S}$ cogenerates $\der{A}$.
    \end{itemize}
\end{defn}

It is well known that for a ring $A$, its unbounded derived category $\der{A}$ is generated by the projective $A$-modules and cogenerated by the injective $A$-modules, \cite[Proposition 2.2]{Rickard2019}. Since a localising subcategory is closed under set indexed coproducts and direct summands, it follows immediately that the regular module $A_A$ also generates $\der{A}$.

\begin{remark} \label{remark:(co)generatorsprove(pc)ig}
    Let $A$ be a ring and $M_A$ a generator of $\Mod{A}$ (in the sense that every $A$-module is a quotient of a set indexed coproduct of copies of $M$). Recall that since $M_A$ is a generator every projective $A$-module is a direct summand of a set indexed coproduct of copies of $M_A$. Thus $M_A$ generates $\der{A}$.
    
    Similarly, if $M_A$ is a cogenerator of $\Mod{A}$ (in the sense that every $A$-module is a submodule of a set indexed product of copies of $M_A$). Then every injective $A$-module is a direct summand of a set indexed product of copies of $M_A$, and $M_A$ cogenerates $\der{A}$. In particular, the injective cogenerator $\cogenerator{A} = \Homgroup{\integers}{{_AA}}{\quotient{\rationals}{\integers}{}}$ cogenerates $\der{A}$.
\end{remark}

In this paper we investigate when the derived category is generated as a localising subcategory by the injective modules and as a colocalising subcategory by the projective modules.

\begin{defn}
    Let $A$ be a ring.
    \begin{enumerate}[label = (\roman*)]
        \item If $\lsubinj{A} = \der{A}$ then we say that injectives generate for $A$.
        \item If $\cosubproj{A} = \der{A}$ then we say that projectives cogenerate for $A$.
    \end{enumerate}
\end{defn}

\subsection{Functors}

Many of the results in this paper rely on using functors that preserve the properties defining localising and colocalising subcategories. Since the ideas are mentioned often, we collate them here. 

\begin{defn}[(Pre)image]
    Let $A$ and $B$ be rings and $\function{F}{\der{A}}{\der{B}}$ be a triangle functor. 
    \begin{itemize}
        \item Let $\cat{C}_B$ be a full triangulated subcategory of $\der{B}$. The preimage of $\cat{C}_B$ under $F$ is the smallest full subcategory of $\der{A}$ consisting of complexes $X$ such that $F(X)$ is in $\cat{C}_B$. (Note that this subcategory is a triangulated subcategory of $\der{A}$.)
        \item Let $\cat{C}_A$ be a full triangulated subcategory of $\der{A}$. The image of $F$ applied to $\cat{C}_A$ is the smallest full triangulated subcategory of $\der{B}$ containing the complexes $F(X)$ for all complexes $X$ in $\cat{C}_A$. Denote the image of $F$ by $\im{F}$.
    \end{itemize}
\end{defn}

\begin{lemma} \label{prop:lsub}
    Let $A$ and $B$ be rings and let $\function{F}{\der{A}}{\der{B}}$ be a triangle functor.
	\begin{enumerate}[label=(\roman*)]
	    \item If $F$ preserves set indexed coproducts then the preimage of a localising subcategory of $\der{B}$ is a localising subcategory of $\der{A}$.
	    \item If $F$ preserves set indexed products then the preimage of a colocalising subcategory of $\der{B}$ is a colocalising subcategory of $\der{A}$.
	\end{enumerate}
\end{lemma}

\begin{proof}
    The result follows immediately by applying the definitions of localising and colocalising subcategories.
\end{proof}

\begin{prop} \label{prop:imageoffunctor}
	Let $A$ and $B$ be rings and $\function{F}{\der{A}}{\der{B}}$ be a triangle functor. Let $\cat{S}$ and $\cat{T}$ be classes of complexes in $\der{A}$ and $\der{B}$ respectively.
	\begin{enumerate}[label=(\roman*)]
	    \item Suppose that $\cat{S}$ generates $\der{A}$. If $F$ preserves set indexed coproducts and $F(S)$ is in $\lsub{B}{\cat{T}}$ for all $S$ in $\cat{S}$, then $\im{F}$ is a subcategory of $\lsub{B}{\cat{T}}$.
	    \item Suppose that $\cat{S}$ cogenerates $\der{A}$. If $F$ preserves set indexed products and $F(S)$ is in $\cosub{B}{\cat{T}}$ for all $S$ in $\cat{S}$, then $\im{F}$ is a subcategory of $\cosub{B}{\cat{T}}$.
	\end{enumerate}
\end{prop}

\begin{proof}
    \begin{enumerate}[label = (\roman*)]
        \item Suppose that $\function{F}{\der{A}}{\der{B}}$ preserves set indexed coproducts and that $F(S)$ is in $\lsub{B}{\cat{T}}$ for all $S$ in $\cat{S}$. By Lemma~\ref{prop:lsub}, the preimage of $\lsub{B}{\cat{T}}$ under $F$ is a localising subcategory. Furthermore, the preimage contains $\cat{S}$ so it also contains $\lsub{A}{\cat{S}} = \der{A}$. Thus $F(X)$ is in $\lsub{B}{\cat{T}}$ for all $X \in \der{A}$.
	    
	    \item The second statement follows similarly.
    \end{enumerate}
	
\end{proof}

\subsection{Adjoint functors} \label{subsection:adjointfunctors}

Adjoint pairs of functors are particularly rich in the various properties they preserve. These properties can be characterised using homomorphisms groups in the derived category. Some of these well known results can be found in \cite[Proof of Proposition 8.1]{Rickard1989} and \cite[Proof of Theorem 1]{Koenig1991}.
	
	\begin{remark} \label{remark:cohomologyindicators}
	    Let $A$ be a ring. Let $X \in \der{A}$ and $n \in \integers$. Recall that, since $A$ is a projective generator of $\Mod{A}$, $\dHom{A}{A}{X[n]} =0$ if and only if $\cohomology{n}{X}=0$. Similarly, since $\cogenerator{A}$ is an injective cogenerator of $\Mod{A}$, $\dHom{A}{X}{\cogenerator{A}[n]}=0$ if and only if $\cohomology{-n}{X}=0$.
	\end{remark}

	\begin{lemma} \label{lemma:homgroupcriteriabded}
    	Let $A$ be a ring. 
    	\begin{enumerate}[label = (\roman*)]
    	    \item A complex $X \in \der{A}$ is bounded in cohomology if and only if for all compact objects $C \in \der{A}$ we have that $\dHom{A}{C}{X[n]}=0$ for all but finitely many $n \in \integers$.
    	    
    	    \item A complex $I \in \der{A}$ is quasi-isomorphic to a bounded complex of injectives if and only if for all complexes $X \in \der{A}$ that are bounded in cohomology, $\dHom{A}{X}{I[n]}=0$ for all but finitely many $n \in \integers$.
    	    
    	    \item A complex $P \in \der{A}$ is quasi-isomorphic to a bounded complex of projectives if and only if for all complexes $X \in \der{A}$ that are bounded in cohomology, $\dHom{A}{P}{X[n]}=0$ for all but finitely many $n \in \integers$.
    	\end{enumerate}
	\end{lemma}	
	
	\begin{proof}
	    \begin{enumerate}[label = (\roman*)]
	        \item Let $X \in \der{A}$. Suppose that for all compact objects, $C \in \der{A}$, we have that $\dHom{A}{C}{X[n]}$ is zero for all but finitely many $n \in \integers$. Since $A$ is compact, $\cohomology{n}{X} \cong \dHom{A}{A}{X[n]}$ is zero for all but finitely many $n \in \integers$.
	    
	        Now suppose that $X \in \der{A}$ is bounded in cohomology. Then there exists $Y \in \der{A}$ such that $Y$ is bounded and $X$ is quasi-isomorphic to $Y$. Let $C \in \der{A}$ be a compact object. Then $C$ is quasi-isomorphic to a bounded complex of finitely generated projectives $P \in \der{A}$. Hence, for all $n \in \integers$,
        	\begin{equation*}
        	   \dHom{A}{C}{X[n]} \cong \kHom{A}{P}{Y[n]}.
        	\end{equation*}
        	    
        	Since both $P$ and $Y$ are bounded, $\kHom{A}{P}{Y[n]} = 0$ for all but finitely many $n \in \integers$.
        	 
        	\item Suppose that $I \in \der{A}$ is quasi-isomorphic to a bounded complex of injectives $J \in \der{A}$. Let $X \in \der{A}$ be bounded in cohomology. Then there exists $Y \in \der{A}$ such that $Y$ is bounded and $X$ is quasi-isomorphic to $Y$. Thus
        	\begin{equation*}
        	    \dHom{A}{X}{I[n]} \cong \kHom{A}{Y}{J[n]},
        	\end{equation*}
        	for all $n \in \integers$. Moreover, both $Y$ and $J$ are bounded so $\kHom{A}{Y}{J[n]}=0$ for all but finitely many $n \in \integers$.
	    
	        Now let us consider the converse. Let $Z \in \der{A}$. Suppose that for all $X \in \der{A}$, such that $X$ is bounded in cohomology, we have that $\dHom{A}{X}{Z[n]}=0$ for all but finitely many $n \in \integers$. Then, taking $X=A$, we have that $\cohomology{n}{Z} \cong \dHom{A}{A}{Z[n]}=0$ for all but finitely many $n \in \integers$. Consequently, $Z$ has a $\cat{K}$-injective resolution $I = \cx{I}{d}{i}$ which is bounded below and has nonzero cohomology in only finitely many degrees. Hence there exists $N \in \integers$ such that for all $n > N$, $\cohomology{n}{I} = 0$ and for all $A$-modules $M$,
	        \begin{equation*}
	            0 = \dHom{A}{M}{Z[n]} \cong \kHom{A}{M}{I[n]}.
	        \end{equation*}
	        In particular,
	        \begin{equation*}
	            \kHom{A}{\kernel{d^n}}{I[n]} = 0,
	        \end{equation*}
	        for all $n > N$. Hence the epimorphism $\function{d^{n-1}}{I^{n-1}}{\im{d^{n-1}} \cong \kernel{d^n}}$ splits. Consequently, $\kernel{d^n}$ is an injective $A$-module and the good truncation
	        \begin{equation*}
	            \goodtrunabove{n}{I} = \dots \rightarrow 0 \rightarrow I^0 \rightarrow I^1 \rightarrow \dots \rightarrow I^{n-1} \rightarrow \kernel{d^n} \rightarrow 0 \rightarrow \dots,
	        \end{equation*}
	        is a bounded complex of injectives that is quasi-isomorphic to $I$.
	        
	        \item This follows similarly to part (ii).
	    \end{enumerate}
	\end{proof}
	
	\begin{lemma} \label{lemma:homgroupcriteriabdedaboveandbelow}
	    Let $A$ be a ring and let $X \in \der{A}$.
	    \begin{enumerate}[label = (\roman*)]
	        \item The following are equivalent:
        	    \begin{enumerate}[label = (\alph*)]
        	        \item $X$ is bounded above in cohomology,
        	        \item For all compact objects $C \in \der{A}$ there exists $N \in \integers$ such that $\dHom{A}{C}{X[n]}=0$ for all $n > N$,
        	        \item For all $Y \in \der{A}$ that are bounded in cohomology there exists $N \in \integers$ such that $\dHom{A}{X}{Y[n]}=0$ for all $n < N$.
        	    \end{enumerate}
            \item The following are equivalent:
                \begin{enumerate}
                    \item $X$ is bounded below in cohomology,
                    \item For all compact objects $C \in \der{A}$ there exists $N \in \integers$ such that $\dHom{A}{C}{X[n]}=0$ for all $n < N$,
                    \item For all $Y \in \der{A}$ that are bounded in cohomology there exists $N \in \integers$ such that $\dHom{A}{Y}{X[n]}=0$ for all $n < N$.
                \end{enumerate}
	    \end{enumerate}
	\end{lemma}
	
	\begin{proof}
	    \begin{enumerate}[label = (\roman*)]
	        \item The equivalence of (a) and (b) is similar to the proof of Lemma~\ref{lemma:homgroupcriteriabded}~(i).
	        
	        Now, we show that $(c)$ implies $(a)$. Suppose that for all $Y \in \der{A}$ that are bounded in cohomology there exists $N \in \integers$ such that $\dHom{A}{X}{Y[n]}=0$ for all $n < N$. Then, taking $Y = \cogenerator{A}$, we have that $\dHom{A}{X}{\cogenerator{A}[n]}=0$ for all $n < N$. Hence $\cohomology{-n}{X}=0$ for all $n < N$.
	        
	        Finally, we show that $(a)$ implies $(c)$. Suppose that $X \in \der{A}$ is bounded above in cohomology. Then $X$ has a bounded above $\cat{K}$-projective resolution $P \in \der{A}$. Let $Y \in \der{A}$ be bounded in cohomology. Then there exists $Z \in \der{A}$ such that $Z$ is bounded and $Y$ is quasi-isomorphic to $Z$. Thus
	        \begin{equation*}
	            \dHom{A}{X}{Y[n]} \cong \kHom{A}{P}{Z[n]},
	        \end{equation*}
	        for all $n \in \integers$. Since $P$ is bounded above and $Z$ is bounded there exists $N \in \integers$ such that $\kHom{A}{P}{Z[n]}=0$ for all $n < N$.
        	
        	\item This follows similarly to (i).
	    \end{enumerate}
	\end{proof}

	Since the properties considered in Remark~\ref{remark:cohomologyindicators}, Lemma~\ref{lemma:homgroupcriteriabded} and Lemma~\ref{lemma:homgroupcriteriabdedaboveandbelow} are defined using homomorphism groups they interact well with adjoint functors. Recall that for brevity, we shall often abuse terminology by writing, for example, that a functor preserves bounded complexes of projectives when we mean that it preserves the property of being quasi-isomorphic to a bounded complex of projectives.
	
	\begin{lemma} \label{lemma:adjointfunctorspreserve}
	    Let $A$ and $B$ be rings. Let $\function{F}{\der{A}}{\der{B}}$ and $\function{G}{\der{B}}{\der{A}}$ be triangle functors such that $(F,G)$ is an adjoint pair.
	    \begin{enumerate}[label=(\roman*)]
	        \item If $G$ preserves set indexed coproducts then $F$ preserves compact objects.
	        \item If $F$ preserves compact objects then $G$ preserves complexes bounded (above or below) in cohomology.
	        \item If $F$ preserves complexes bounded in cohomology then $G$ preserves bounded complexes of injectives and complexes bounded below in cohomology.
	        \item If $G$ preserves bounded complexes of injectives then $F$ preserves complexes bounded (above or below) in cohomology.
	        \item If $G$ preserves complexes bounded in cohomology then $F$ preserves bounded complexes of projectives and complexes bounded above in cohomology.
	        \item If $F$ preserves bounded complexes of projectives then $G$ preserves complexes bounded (above or below) in cohomology. 
	    \end{enumerate}
	\end{lemma}
	
	\begin{proof}
	    \begin{enumerate}[label = (\roman*)]
	        \item Let $C \in \der{A}$ be a compact object and let $\{X_i\}_{i \in I} \in \der{A}$ for an index set $I$.
	        Since $C$ is compact and $G$ preserves set indexed coproducts there are natural isomorphisms
	        \begin{align*}
	            \dHom{B}{F(C)}{\ds{i}{I}{X_i}} &\cong \dHom{A}{C}{\ds{i}{I}{G(X_i)}},
	            \\
	            &\cong \ds{i}{I}{\dHom{A}{C}{G(X_i)}},
	            \\
	            &\cong \ds{i}{I}{\dHom{A}{F(C)}{X_i}}.
	        \end{align*}
	        Thus $F(C)$ is compact.
	   \end{enumerate}
	   
        The remaining claims follow from Remark~\ref{remark:cohomologyindicators}, Lemma~\ref{lemma:homgroupcriteriabded} and Lemma~\ref{lemma:homgroupcriteriabdedaboveandbelow} and the adjunction
        \begin{equation*}
            \dHom{A}{X}{G(Y)[n]} \cong \dHom{B}{F(X)}{Y[n]},
        \end{equation*}
        for appropriate choices of $X$ or $Y$.
        \begin{enumerate}[label = (\roman*), resume]
            \item Let $X$ run over all compact objects, and apply Lemma~\ref{lemma:homgroupcriteriabded}~(i), Lemma~\ref{lemma:homgroupcriteriabdedaboveandbelow}~(i) and Lemma~\ref{lemma:homgroupcriteriabdedaboveandbelow}~(ii).
            \item Let $X$ run over all complexes bounded in cohomology, and apply Lemma~\ref{lemma:homgroupcriteriabded}~(ii) and Lemma~\ref{lemma:homgroupcriteriabdedaboveandbelow}~(ii).
            \item Let $Y = \cogenerator{B}$ and apply Remark~\ref{remark:cohomologyindicators}.
            \item Let $Y$ run over all complexes bounded in cohomology, and apply Lemma~\ref{lemma:homgroupcriteriabded}~(iii) and Lemma~\ref{lemma:homgroupcriteriabdedaboveandbelow}~(i).
            \item Let $X=A$ and apply Remark~\ref{remark:cohomologyindicators}.
        \end{enumerate}
	\end{proof}

\section{Tensor product algebra} \label{section:tensorproductalgebra}

The first ring construction we consider is the tensor product of two finite dimensional algebras $A$ and $B$, over a field $k$. In particular, we prove that if injectives generate for the two algebras then injectives generate for their tensor product and similarly for projectives cogenerating. Firstly, we recall a description of the injective and projective modules for a tensor product algebra.

\begin{lemma} \cite[Chapter IX, Proposition 2.3]{Cartan1956}, \cite[Lemma 3.1]{Xi2000} \label{lem:tensorproduct-findimfield-inj&projmodules}
	Let $A$ and $B$ be finite dimensional algebras over a field $k$. Let $M_A$ be an $A$-module and $N_B$ be a $B$-module.
	\begin{enumerate}[label=(\roman*)]
	    \item If $M_A$ is a projective $A$-module and $N_B$ is a projective $B$-module then $\tens{M}{k}{N}$ is a projective $\left( \tens{A}{k}{B} \right)$-module.
	    \item If $M_A$ is an injective $A$-module and $N_B$ is an injective $B$-module then $\tens{M}{k}{N}$ is an injective $\left( \tens{A}{k}{B} \right) $-module.
	\end{enumerate}
\end{lemma}

Notice that the structure of these modules is functorial in either argument. For a $B$-module $N_B$ define $F_N \coloneqq \function{\tens{-}{k}{N}}{\Mod{A}}{\Mod{(\tens{A}{k}{B})}}$. Similarly, for an $A$-module $M_A$ define $G_M \coloneqq \function{\tens{M}{k}{-}}{\Mod{B}}{\Mod{(\tens{A}{k}{B})}}$. Since $k$ is a field the functors $F_N$ and $G_M$ are exact. Hence these functors extend to triangle functors $\function{F_N}{\der{A}}{\der{\tens{A}{k}{B}}}$ and $ \function{G_M}{\der{B}}{\der{\tens{A}{k}{B}}}$.

\begin{prop} \label{prop:tensorproduct-findimfield-ig&pc}
    Let $A$ and $B$ be finite dimensional algebras over a field $k$. 
    \begin{enumerate}[label=(\roman*)]
        \item If injectives generate for both $A$ and $B$ then injectives generate for $\tens{A}{k}{B}$.
        \item If projectives cogenerate for both $A$ and $B$ then projectives cogenerate for $\tens{A}{k}{B}$.
    \end{enumerate}
\end{prop}

\begin{proof}
    \begin{enumerate}[label = (\roman*)]
        \item Since $A$ is a finite dimensional algebra over a field $k$ every injective $A$-module is a direct summand of a set indexed coproduct of copies of $\dual{A} = \Homgroup{k}{A}{k}$. Thus $\lsubinj{A} = \lsub{A}{\dual{A}}$. Similarly, $\lsubinj{B} = \lsub{B}{\dual{B}}$.
        
        Denote $C \coloneqq \tens{A}{k}{B}$. By Lemma~\ref{lem:tensorproduct-findimfield-inj&projmodules}, $F_{\dual{B}}(\dual{A}) = \tens{\dual{A}}{k}{\dual{B}}$ is an injective $C$-module. Moreover, $F_{\dual{B}}$ preserves set indexed coproducts. Suppose that injectives generate for $A$. Then, by Proposition~\ref{prop:imageoffunctor}, $\im{F_{\dual{B}}}$ is a subcategory of $\lsubinj{C}$. In particular, $F_{\dual{B}}(A) = \tens{A}{k}{\dual{B}}$ is in $\lsubinj{C}$.
        
        Now consider the functor $G_{A} \coloneqq \tens{A}{k}{-}$. By the previous argument  
        \begin{equation*}
            G_{A}(\dual{B}) = \tens{A}{k}{\dual{B}} = F_{\dual{B}}(A) \in \lsubinj{C}.
        \end{equation*}
        Moreover, $G_{A}$ preserves set indexed coproducts. Suppose that injectives generate for $B$. Then, by Proposition~\ref{prop:imageoffunctor}, $\im{G_{A}}$ is a subcategory of $\lsubinj{C}$. In particular, $G_A(B) = \tens{A}{k}{B} = C$ is in $\lsubinj{C}$. Consequently, $\lsub{C}{C} = \der{C}$ is a subcategory of $\lsubinj{C}$ and injectives generate for $C = \tens{A}{k}{B}$. 
        
        \item The projectives cogenerate statement follows similarly by considering $F_B := \tens{-}{k}{B}$ and $G_{\dual{A}} := \tens{\dual{A}}{k}{-}$.
    \end{enumerate}
    
\end{proof}

The converse to Proposition~\ref{prop:tensorproduct-findimfield-ig&pc} is shown as an application of the results about ring homomorphisms considered in Section~\ref{section:ringhomomorphisms}. In particular, the converse statement follows immediately from Lemma~\ref{lem:ring_ext_gen_right_implies_gen_left}.

\section{Separable equivalence} \label{section:equivalences}

Rickard proved that if two algebras are derived equivalent then injectives generate for one if and only if injectives generate for the other \cite[Theorem 3.4]{Rickard2019}. Here we show the statement extends to separable equivalence of rings, defined by Linckelmann \cite[Section 3]{Linckelmann2011}. In particular, we prove the result using separably dividing rings \cite[Section 2]{BerghErdmann2011}.

\begin{defn}[Separably dividing rings.] \label{defn:separablydividingrings}
    Let $A$ and $B$ be rings. Then $B$ separably divides $A$ if there exist bimodules $_AM_B$ and $_BN_A$ such that:
    \begin{enumerate}[label = (\roman*)]
        \item The modules $_AM$, $M_B$, $_BN$ and $N_A$ are all finitely generated projectives.
        \item There exists a bimodule $_BY_B$ such that $\tens{_BN}{A}{M_B}$ and ${_BB_B} \oplus {_BY_B}$ are isomorphic as $(B,B)$-bimodules.
    \end{enumerate}
\end{defn}

\begin{prop} \label{prop:separableequivalence}
    Let $A$ and $B$ be rings such that $B$ separably divides $A$. 
    \begin{enumerate}[label = (\roman*)]
        \item If injectives generate for $A$ then injectives generate for $B$.
        \item If projectives cogenerate for $A$ then projectives cogenerate for $B$.
    \end{enumerate}
\end{prop}

\begin{proof}
    \begin{enumerate}[label = (\roman*)]
        \item Let $_AM_B$, $_BN_A$ and $_BY_B$ be as in Definition~\ref{defn:separablydividingrings}. Consider the adjoint functors
        \begin{align*}
            \function{\tens{-}{B}{N}}{\Mod{B}}{\Mod{A}},
            \\
            \function{\Homgroup{A}{N}{-}}{\Mod{A}}{\Mod{B}}.
        \end{align*}
        Since both $_BN$ and $N_A$ are projective, $\tens{-}{B}{N}$ and $\Homgroup{A}{N}{-}$ are exact and extend to triangle functors. As $\Homgroup{A}{N}{-}$ has an exact left adjoint it preserves injective modules. Furthermore, the module $N_A$ is a finitely generated projective so $\Homgroup{A}{N}{-}$ also preserves set indexed coproducts. 
        
        Suppose that injectives generate for $A$. Since $\Homgroup{A}{N}{-}$ preserves injective modules and coproducts its image is a subcategory of $\lsubinj{B}$ by Proposition~\ref{prop:imageoffunctor}. By adjunction 
        \begin{equation*}
            \Hom{A}{}{N}{}{\Hom{B}{A}{M}{}{B}} \cong \Hom{B}{}{\tens{N}{A}{M}}{}{B}{},
        \end{equation*}
        as right $B$-modules and so $\Hom{A}{}{\tens{N}{A}{M}}{}{B}$ is in $\lsubinj{B}$. Since $\tens{_BN}{A}{M_B} \cong {_BB_B} \oplus {_BY_B}$ as $(B,B)$-bimodules,
        \begin{equation*}
            \Hom{B}{}{\tens{N}{A}{M}}{}{B} \cong B \oplus \Hom{B}{}{Y}{}{B},
        \end{equation*}
        as right $B$-modules. Moreover, localising subcategories are closed under direct summands so $B$ is in $\lsubinj{B}$ and injectives generate for $B$.
        
        \item Suppose that projectives cogenerate for $A$. Since $_AM$ and $M_B$ are finitely generated projective modules, $\tens{-}{A}{M_B}$ preserves set indexed products and projective modules. Hence $\im{\tens{-}{A}{M_B}}$ is a subcategory of $\cosubproj{B}$ by Proposition~\ref{prop:imageoffunctor}. Since $\cogenerator{B}$ is a direct summand of $\tens{\left(\tens{\cogenerator{B}}{B}{N}\right)}{A}{M}$, $\cogenerator{B}$ is in $\cosubproj{B}$ so projectives cogenerate for $B$.
    \end{enumerate}

\end{proof}

\begin{defn}[Separable Equivalence] \cite[Definition 3.1]{Linckelmann2011}
    Let $A$ and $B$ be rings. Then $A$ and $B$ are separably equivalent if $A$ separably divides $B$ via bimodules ${_AM_B}$ and ${_BN_A}$ and $B$ separably divides $A$ via the same bimodules ${_AM_B}$ and ${_BN_A}$. 
\end{defn}

\begin{eg}
    Let $G$ be a finite group. Let $k$ be a field of positive characteristic $p$ and $H$ be a Sylow p-subgroup of $G$. Then the group algebras $kG$ and $kH$ are separably equivalent \cite[Section 3]{Linckelmann2011}.
\end{eg}

\begin{cor}
    Let $A$ and $B$ be separably equivalent rings.
    \begin{enumerate}[label = (\roman*)]
        \item Injectives generate for $A$ if and only if injectives generate for $B$.
        \item Projectives cogenerate for $A$ if and only if projectives cogenerate for $B$.
    \end{enumerate}
\end{cor}

\begin{proof}
    Since $A$ and $B$ are separably equivalent, $A$ separably divides $B$ and $B$ separably divides $A$.
\end{proof}

\section{Ring homomorphisms} \label{section:ringhomomorphisms}
   
Ring homomorphisms give rise to a triple of functors which interact well with both injectives generate and projectives cogenerate statements. In this section we provide sufficient conditions on a ring extension $\function{f}{B}{A}$ such that if injectives generate for $A$ then injectives generate for $B$, and vice versa. In particular, we focus on Frobenius extensions \cite{Kasch1954, Nakayama1960}, almost excellent extensions \cite{Xue1996} and trivial extensions. In Subsection~\ref{subsection:arrowremoval} we apply the results in this section to the arrow removal operation defined in \cite[Section 4]{Green2018}. 
   
Recall that for a ring homomorphism $\function{f}{B}{A}$ there exist three functors between the module categories of $A$ and $B$, denoted as follows,
        \begin{itemize}
            \item Induction, $\ind{B}{A}{} \coloneqq \function{\tens{-}{B}{A_A}}{\Mod{B}}{\Mod{A}}$,
            \item Restriction, $\res{B}{A}{} \coloneqq \function{\Hom{A}{B}{A}{}{-}}{\Mod{A}}{\Mod{B}}$,
            \item Coinduction, $\coind{B}{A}{} \coloneqq \function{\Hom{B}{}{_AA}{}{-}}{\Mod{B}}{\Mod{A}}$.
        \end{itemize}
    
    Note that both $(\ind{B}{A}{},\res{B}{A}{})$ and $(\res{B}{A}{},\coind{B}{A}{})$ are adjoint pairs of functors.
    
    \begin{remark} \label{remark:(co)indpreservesboundedcomplexes}
        Recall that, if $_BA$ has finite flat dimension and $M \in \Mod{B}$ then $\tor{B}{i}{M}{A}=0$ for all but finitely many $i \in \integers$ so $\lind{B}{A}{M}$ is bounded in cohomology. Hence, as $\lind{B}{A}{}$ is a triangle functor, $\lind{B}{A}{}$ preserves complexes bounded in cohomology. Similarly, if $A_B$ has finite projective dimension then $\ext{i}{B}{A}{M}=0$ for all but finitely many $i \in \integers$ and $\rcoind{B}{A}{}$ preserves complexes bounded in cohomology.
    \end{remark}

    \begin{lemma} \label{lem:ring_ext_gen_right_implies_gen_left}
        Let $A$ and $B$ be rings with a ring homomorphism $\function{f}{B}{A}$.
        \begin{enumerate}[label = (\roman*)]
            \item Suppose that $_BA$ has finite flat dimension as a left $B$-module and that $\res{B}{A}{\Mod{A}}$ generates $\der{B}$. If injectives generate for $A$ then injectives generate for $B$.
            \item Suppose that $A_B$ has finite projective dimension as a right $B$-module and that $\res{B}{A}{\Mod{A}}$ cogenerates $\der{B}$. If projectives cogenerate for $A$ then projectives cogenerate for $B$.
        \end{enumerate}
    \end{lemma}

    \begin{proof}
        \begin{enumerate}[label = (\roman*)]
            \item If $_BA$ has finite flat dimension as a left $B$-module, then $\lind{B}{A}{}$ preserves complexes bounded in cohomology by Remark~\ref{remark:(co)indpreservesboundedcomplexes}. Thus, by Lemma~\ref{lemma:adjointfunctorspreserve}, $\res{B}{A}{}$ preserves bounded complexes of injectives. Furthermore, $\res{B}{A}{}$ preserves set indexed coproducts. Hence if injectives generate for $A$ then $\im{\res{B}{A}{}}$ is a subcategory of $\lsubinj{B}$, by Proposition~\ref{prop:imageoffunctor}. Thus, as $\res{B}{A}{\Mod{A}}$ generates $\der{B}$, injectives generate for $B$.
            
            \item The second statement follows similarly. In particular, if $A_B$ has finite projective dimension as right $B$-module, then $\rcoind{B}{A}{}$ preserves complexes bounded in cohomology by Remark~\ref{remark:(co)indpreservesboundedcomplexes}. Then we apply Lemma~\ref{lemma:adjointfunctorspreserve} and Proposition~\ref{prop:imageoffunctor} to $\res{B}{A}{}$.
        \end{enumerate}
    
    \end{proof}
    
    There are many ways that $\res{B}{A}{\Mod{A}}$ could generate $\der{B}$. If $A_B$ is a generator of $\Mod{B}$ (in the sense that every $B$-module is a quotient of a set indexed coproduct of copies of $A_B$), then, by Remark~\ref{remark:(co)generatorsprove(pc)ig}, $A_B$ generates $\der{B}$. There are many examples of familiar ring extensions which satisfy both this property and the conditions of Lemma~\ref{lem:ring_ext_gen_right_implies_gen_left}, some of which we list here.
    
    \begin{itemize}
        \item \underline{Tensor product algebra.}
        
        Let $A$ and $B$ be finite dimensional algebras over a field $k$. Then the tensor product algebra $\tens{A}{k}{B}$ is an extension of both $A$ and $B$. Consider the ring homomorphism $\function{f}{A}{\tens{A}{k}{B}}$ with $f(a) \coloneqq \tens{a}{k}{1_B}$ for all $a \in A$. In particular, $_A (\tens{A}{k}{B})$ considered as a left $A$-module is a set indexed coproduct of copies of $_AA$, one for each basis element of $_kB$. Hence $_A (\tens{A}{k}{B})$ is flat as a left $A$-module. Furthermore, $(\tens{A}{k}{B})_A$ considered as a right $A$-module is a set indexed coproduct of copies of $A_A$. Thus $(\tens{A}{k}{B})_A$ is a generator of $\Mod{A}$.
        
        \item \underline{Frobenius extensions.}
        
        Kasch \cite{Kasch1954} defined a generalisation of a Frobenius algebra called a free Frobenius extension. Nakayama and Tsuzuku \cite{Nakayama1960} generalised this definition further to a Frobenius extension.

        \begin{defn}[(Free) Frobenius extension]
        Let $A$ and $B$ be rings and $\function{f}{B}{A}$ a ring homomorphism. Then $A$ is a (free) Frobenius extension of $B$ if the following are satisfied.
            \begin{itemize}
                \item The module $A_B$ is a finitely generated projective (respectively free) $B$-module.
                \item The bimodule $\Hom{B}{A}{A}{B}{B}$ is isomorphic as a $(B,A)$-bimodule to $_BA_A$.
            \end{itemize}
        \end{defn}
    
        The definition of a Frobenius extension implies that the two functors, $\ind{B}{A}{}$ and $\coind{B}{A}{}$ are isomorphic. Thus $\ind{B}{A}{}$ is exact, and so $_BA$ is flat.

            \begin{eg} \label{eg:frobeniusextensions}
            There are well known examples of Frobenius extensions which have the property that $A_B$ is a generator of $\Mod{B}$ and that $(\cogenerator{A})_B$ is a cogenerator of $\Mod{B}$.
            
            \begin{itemize}
                \item Free Frobenius extensions
                
                Since $A_B$ is a finitely generated free $B$-module, $A_B$ is a generator of $\Mod{B}$. Moreover, since ${_BA_A} \cong \Homgroup{B}{{_AA}}{{_BB}}$ as $(B,A)$-bimodules, ${_BA}$ is free as a left $B$-module and so $\cogenerator{B}_B$ is a direct summand of $(\cogenerator{A})_B$. Thus $(\cogenerator{A})_B$ is a cogenerator of $\Mod{B}$.
                
                \item Strongly $G$-graded rings for a finite group $G$. \cite[Example B]{Bell1993}.
                
                Let $G$ be a finite group and $A$ be a ring graded by $G$. Then $A$ is strongly graded by $G$ if $A_gA_h = A_{gh}$ for all $g, h \in G$. Let $1$ be the identity element of $G$. Then $A$ is a Frobenius extension of $A_1$. Moreover, $A_1$ is a direct summand of $A$ as an $(A_1, A_1)$-bimodule. Thus $A_{A_1}$ is a generator of $\Mod{A_1}$ and $(\cogenerator{A})_{A_1}$ is a cogenerator of $\Mod{A_1}$ This collection of graded rings includes skew group algebras, smash products and crossed products for finite groups.
                    
            \end{itemize}
            \end{eg}
        
        \item \underline{Almost excellent extensions.}
            
            Almost excellent extensions were defined by Xue \cite{Xue1996} as a generalisation of excellent extensions first introduced by Passman \cite{Passman1977}. The interaction of excellent extensions with various properties of rings has been studied in \cite{Huang2012}.

            \begin{defn}[Almost excellent extension]
                Let $A$ and $B$ be rings. Then $A$ is an almost excellent extension of $B$ if the following hold.
                \begin{itemize}
                    \item There exist $a_1,a_2,...,a_n \in A$ such that $A = \sum_{i=1}^n a_iB$ and $a_iB = Ba_i$ for all $1 \leq i \leq n$.
                    \item $A$ is right $B$-projective.
                    \item $_B A$ is flat and $A_B$ is projective.
                \end{itemize}
            \end{defn}

            Recall the definition of right $B$-projective rings.
            
            \begin{defn}[Right $B$-projective]
                Let $A$ and $B$ be rings and $\function{f}{B}{A}$ be a ring homomorphism. A short exact sequence of $A$-modules
                \begin{equation*}
                    \ses{L_A}{K_A}{N_A}{f}{g},
                \end{equation*}
                is an $(A,B)$-exact sequence if it splits as a short exact sequence of right $B$-modules. 
                
                If every $(A,B)$-exact sequence also splits as a short exact sequence of right $A$-modules then $A$ is right $B$-projective.
            \end{defn}

            By definition $_BA$ is flat and $A_B$ is projective, thus to apply Lemma~\ref{lem:ring_ext_gen_right_implies_gen_left}, all that is left to show is that $\res{B}{A}{\Mod{A}}$ generates and cogenerates $\der{B}$. By \cite[Corollary 4]{Soueif1987}, $\coind{B}{A}{M}= 0$ if and only if $M = 0$ for all $M \in \Mod{B}$. Thus $\Homgroup{B}{\res{B}{A}{A}}{M}=0$ if and only if $M=0$ by adjunction. Since $A_B$ is projective this is equivalent to $A_B$ being a generator of $\Mod{B}$. Similarly, $(\cogenerator{A})_B$ is a cogenerator for $\Mod{B}$ since $\ind{B}{A}{M}=0$ if and only if $M=0$ \cite[Proposition 2.1]{Shamsuddin1992}.
            
        \item \underline{Trivial extension ring.}
        
            \begin{defn}[Trivial extension ring]
                Let $B$ be a ring and ${_BM_B}$ be a $(B,B)$-bimodule. The trivial extension of $B$ by $M$, denoted by $\triext{B}{M}$, is the ring with elements $(b,m) \in B \oplus M$, addition defined by,
                \begin{equation*}
                (b,m) + (b',m') \coloneqq (b+b', m+m'),
                \end{equation*}
                and multiplication defined by,
                \begin{equation*}
                (b,m)(b',m') \coloneqq (bb',bm'+mb').
                \end{equation*}
            \end{defn}

        Let $A \coloneqq \triext{B}{M}$. Then there exists a ring homomorphism $\function{f}{B}{A}$, defined by $f(b) = (b,0)$for all $b \in B$. Note that $_BA$ is isomorphic to $B \oplus M$ as a left $B$-module, thus $_BA$ has finite flat dimension as a left $B$-module if and only if $_BM$ has finite flat dimension as a left $B$-module. Furthermore, $A_B$ is isomorphic to $B \oplus M$ as a right $B$-module and so $A_B$ is a generator of $\Mod{B}$. Hence Lemma~\ref{lem:ring_ext_gen_right_implies_gen_left}~(i)  applies if ${_BM}$ has finite flat dimension. Similarly, Lemma~\ref{lem:ring_ext_gen_right_implies_gen_left}~(ii) applies if $M_B$ has finite projective dimension, since $\cogenerator{B}_B$ is a direct summand of $(\cogenerator{A})_B$.

        \begin{eg}
            \begin{itemize}
                \item Let $A$ be a ring, then $\triext{A}{A}$ is isomorphic to $\quotient{A[x]}{\left< x^2 \right>}{}$.
                \item Let $A$ and $B$ be rings with $_AM_B$ an $(A,B)$-bimodule. Then the triangular matrix ring $\twobytwo{A}{M}{0}{B}$ is isomorphic to $\triext{\left( A \times B \right)}{M}$.
                \item Green, Psaroudakis and Solberg \cite{Green2018} use trivial extension rings to define an operation on quiver algebras called arrow removal. This operation is considered in Subsection \ref{subsection:arrowremoval}. 
            \end{itemize}
        \end{eg}

    \end{itemize}
    
    The examples included above satisfy Lemma~\ref{lem:ring_ext_gen_right_implies_gen_left} since $A_B$ generates $\Mod{B}$. One example of a ring construction which satisfies Lemma~\ref{lem:ring_ext_gen_right_implies_gen_left} without this assumption is a quotient ring $A \coloneqq \quotient{B}{I}{}$ where $I$ is a nilpotent ideal of $B$. Since $A_B$ is annihilated by $I$, $A_B$ does not generate $\Mod{B}$. However, $\res{B}{A}{\Mod{A}}$ does generate $\der{B}$.
    
    \begin{lemma} \label{lemma:nilpotentidealig}
        Let $B$ be a ring and $I$ be a nilpotent ideal of $B$. 
        \begin{enumerate}[label = (\roman*)]
            \item If $_BI$ has finite flat dimension as a left $B$-module and injectives generate for $\quotient{B}{I}{}$ then injectives generate for $B$.
            \item If $I_B$ has finite projective dimension as a right $B$-module and projectives cogenerate for $\quotient{B}{I}{}$ then projectives cogenerate for $B$.
        \end{enumerate}
    \end{lemma}
    
    \begin{proof}
        
        The ring homomorphism we shall use is the quotient map $\function{f}{B}{\quotient{B}{I}{}}$. Let $A \coloneqq \quotient{B}{I}{}$. Then $\res{B}{A}{\Mod{A}}$ consists of the $B$-modules annihilated by $I$. Note that every $B$-module $M$ is an iterated extension of $\quotient{MI^m}{MI^{m+1}}{}$ for $m \in \integers$ via the short exact sequences
        \begin{equation*}
            \ses{MI^{m+1}}{MI^{m}}{\quotient{MI^{m}}{MI^{m+1}}{}}{}{}{}.
        \end{equation*}
        Hence every $B$-module is in the triangulated subcategory of $\der{B}$ generated by $\res{B}{A}{\Mod{A}}$. Thus $\res{B}{A}{\Mod{A}}$ both generates and cogenerates $\der{B}$.
        \begin{enumerate}[label = (\roman*)]
            \item By the short exact sequence of left $B$-modules
            \begin{equation*}
                \ses{I}{B}{A}{}{},
            \end{equation*}
            $_BA$ has finite flat dimension if and only if $_BI$ has finite flat dimension. Hence (i) follows from Lemma~\ref{lem:ring_ext_gen_right_implies_gen_left}~(i).
            \item Similarly, by the short exact sequence of right $B$-modules
            \begin{equation*}
                \ses{I}{B}{A}{}{},
            \end{equation*}
            $A_B$ has finite projective dimension if and only if $I_B$ has finite projective dimension. Hence (ii) follows from Lemma~\ref{lem:ring_ext_gen_right_implies_gen_left}~(ii).
        \end{enumerate}
    \end{proof}
    
    \begin{eg} \label{eg:trivialextension}
        Lemma~\ref{lemma:nilpotentidealig} can be applied to trivial extension rings. In particular, let $S$ be a ring and $M$ be an $(S,S)$-bimodule. Let $R \coloneqq \triext{S}{M}$. Then $(0,M)$ is a nilpotent ideal of $R$ and $S \cong \quotient{R}{(0,M)}{}$. Thus if injectives generate for $S$ and ${_R(0,M)}$ has finite flat dimension then Lemma~\ref{lemma:nilpotentidealig}~(i) applies and injectives generate for $R$. Similarly, Lemma~\ref{lemma:nilpotentidealig}~(ii) applies if $(0,M)_R$ has finite projective dimension.
    \end{eg}

\subsection{Arrow removal} \label{subsection:arrowremoval}

Let $\Lambda \coloneqq \quotient{kQ}{I}{}$ be a path algebra with admissible ideal $I$. Let $\function{a}{v_e}{v_f}$ be an arrow of $Q$ which is not in a minimal generating set of $I$. Then Green, Psaroudakis and Solberg \cite[Section 4]{Green2018} define the algebra obtained from $\Lambda$ by removing the arrow $a$ as $\Gamma \coloneqq \quotient{\Lambda}{\Lambda \bar{a} \Lambda}{}$ where $\bar{a} = a + I$.

\begin{prop} \cite[Proposition 4.4~(iii)]{Green2018} \label{prop:solbergarrowremovaltrivialext}
    Let $\Lambda \coloneqq \quotient{kQ}{I}{}$ be an admissible quotient of the path algebra $kQ$ over a field $k$. Suppose that there are arrows $\function{a_i}{v_{e_i}}{v_{f_i}}$ in $Q$ for $i = 1,2,\dots,t$ which do not occur in a set of minimal generators of $I$ in $kQ$ and $\Hom{\Lambda}{}{e_i\Lambda}{}{f_j\Lambda} = 0$ for all $i$ and $j$ in $\{1,2,\dots,t\}$. Let $\bar{a_i} = a_i + I$ in $\Lambda$. Let $\Gamma = \quotient{\Lambda}{\Lambda \{\bar{a_i}\}_{i=1}^t \Lambda}{}$. Then $\Lambda$ is isomorphic to the trivial extension $\triext{\Gamma}{P}$, where $P = \oplus_{i=1}^t \tens{\Gamma e_i}{k}{f_i \Gamma}$.
\end{prop}

The properties of $P$ in Proposition~\ref{prop:solbergarrowremovaltrivialext} satisfy the assumptions of both Lemma~\ref{lem:ring_ext_gen_right_implies_gen_left} and Lemma~\ref{lemma:nilpotentidealig}. Hence we can use this arrow removal technique for injective generation.

\begin{prop} 
    Let $\Lambda \coloneqq \quotient{kQ}{I}{}$ be an admissible quotient of the path algebra $kQ$ over a field $k$. Suppose that there are arrows $\function{a_i}{v_{e_i}}{v_{f_i}}$ in $Q$ for $i = 1,2,\dots,t$ which do not occur in a set of minimal generators of $I$ in $kQ$ and $\Hom{\Lambda}{}{e_i\Lambda}{}{f_j\Lambda} = 0$ for all $i$ and $j$ in $\{1,2,\dots,t\}$. Let $\bar{a_i} = a_i + I$ in $\Lambda$. Let $\Gamma = \quotient{\Lambda}{\Lambda \{\bar{a_i}\}_{i=1}^t \Lambda}{}$.
    \begin{enumerate}[label = (\roman*)]
        \item Injectives generate for $\Lambda$ if and only if injectives generate for $\Gamma$.
        \item Projectives cogenerate for $\Lambda$ if and only if projectives cogenerate for $\Gamma$.
    \end{enumerate}
\end{prop}

\begin{proof}
    Firstly, note that $P = \oplus_{i=1}^t \tens{\Gamma e_i}{k}{f_i \Gamma}$ is projective as both a left and right $\Gamma$-module. Moreover, $_\Gamma \Lambda _\Gamma$ is isomorphic to $\Gamma \oplus P$ as a $(\Gamma, \Gamma)$-bimodule. Hence $\Lambda$ is projective as both a left and right $\Gamma$-module. Consequently, Lemma~\ref{lem:ring_ext_gen_right_implies_gen_left} applies with $A = \Lambda$ and $B = \Gamma$. So if injectives generate for $\Lambda$ then injectives generate for $\Gamma$ and the equivalent statement for projective cogeneration also holds.
    
    Secondly, there exists a short exact sequence of $(\Lambda,\Lambda)$-bimodules
    \begin{equation*}
        \ses{P}{\Lambda}{\Gamma}{}{}.
    \end{equation*}
    Hence
    \begin{equation*}
        \tens{P}{\Gamma}{P} \rightarrow \tens{\Lambda}{\Gamma}{P} \rightarrow \tens{\Gamma}{\Gamma}{P} \rightarrow 0,
    \end{equation*}
    is an exact sequence of left $\Lambda$-modules. Moreover, by \cite[Proposition 4.6~(vii)]{Green2018}, $\tens{P}{\Gamma}{P} = 0$. Thus
    \begin{equation*}
         \tens{_\Lambda \Lambda}{\Gamma}{P} \cong  \tens{_\Lambda \Gamma}{\Gamma}{P} \cong {_\Lambda P},
    \end{equation*}
    so $_\Lambda P$ is a projective left $\Lambda$-module. Similarly,
    \begin{equation*}
        \tens{P}{\Gamma}{\Lambda_\Lambda} \cong \tens{P}{\Gamma}{\Gamma_\Lambda} \cong P_\Lambda,
    \end{equation*}
    so $P_\Lambda$ is a projective right $\Lambda$-module. Consequently, Lemma~\ref{lemma:nilpotentidealig} applies. In particular, Example~\ref{eg:trivialextension} applies with $R = \Lambda$, ${(0,{_SM_S})} = {(0,{_\Gamma P_\Gamma})} = {_\Lambda P_\Lambda}$ and $S = \quotient{R}{(0,M)}{} = \Gamma$. So if injectives generate for $\Gamma$ then injectives generate for $\Lambda$ and the similar statement holds for projective cogeneration.
\end{proof}

    \subsection{Frobenius extensions and almost excellent extensions}
    
    To prove the converse statement to Lemma~\ref{lem:ring_ext_gen_right_implies_gen_left} for Frobenius extensions and almost excellent extensions we prove a more general result connecting relatively $B$-injective $A$-modules to injective generation.
    
    \begin{defn}[Relatively projective/injective]
        Let $A$ and $B$ be rings with a ring homomorphism $\function{f}{B}{A}$. Let $M_A$ be an $A$-module.
        \begin{itemize}
            \item $M_A$ is relatively $B$-projective if $\Homgroup{A}{M}{-}$ is exact on $(A,B)$-exact sequences.
            \item $M_A$ is relatively $B$-injective if $\Homgroup{A}{-}{M}$ is exact on $(A,B)$-exact sequences.
        \end{itemize}
    \end{defn}
    
    Let $A$ and $B$ be rings with a ring homomorphism $\function{f}{B}{A}$. Then any injective $A$-module, $I$, is relatively $B$-injective since $\Homgroup{A}{-}{I}$ is exact on all short exact sequences of $A$-modules. Similarly, any projective $A$-module is relatively $B$-projective. However, for both Frobenius extensions and almost excellent extensions all projective $A$-modules are relatively $B$-injective. This property can be used to prove the converse statement to Lemma~\ref{lem:ring_ext_gen_right_implies_gen_left} for these ring extensions.
    
    \begin{lemma} \label{lem:ring_ext_relatively_proj_inj_generation}
        Let $A$ and $B$ be rings with a ring homomorphism $\function{f}{B}{A}$.
        \begin{enumerate}[label = (\roman*)]
            \item Suppose that $A_B$ is a finitely generated projective $B$-module and that all projective $A$-modules are relatively $B$-injective. If injectives generate for $B$ then injectives generate for $A$.
            \item Suppose that $_BA$ is a finitely generated projective $B$-module and that all injective $A$-modules are relatively $B$-projective. If projectives cogenerate for $B$ then projectives cogenerate for $A$.
        \end{enumerate}
    \end{lemma}
    
    \begin{proof}
        \begin{enumerate}[label = (\roman*)]
            \item Since $A_B$ is a finitely generated projective, $\coind{B}{A}{}$ is exact and preserves both set indexed coproducts and injective modules. Hence, if injectives generate for $B$, then $\im{\coind{B}{A}{}}$ is a subcategory of $\lsubinj{A}$ by Proposition~\ref{prop:imageoffunctor}. Let $P_A$ be a projective $A$-module. Since $P$ is relatively $B$-injective, $P$ is a direct summand of $\coind{B}{A}{} \circ \res{B}{A}{P}$, \cite[Section 4.1]{Kadison1999}. Thus $\Proj{A}$ is a subcategory of $\lsubinj{A}$ and injectives generate for $A$. 
            
            \item Similarly, if an injective $A$-module $I$ is relatively $B$-projective then $I$ is a direct summand of $\ind{B}{A}{} \circ \res{B}{A}{I}$, \cite[Section 4.1]{Kadison1999}. Moreover, $\ind{B}{A}{}$ is exact and preserves set indexed products as $_BA$ is a finitely generated projective. Thus, if projectives cogenerate for $B$ then the $\im{\ind{B}{A}{}}$ is a subcategory of $\cosubproj{A}$, by Proposition~\ref{prop:imageoffunctor}, and $\Inj{A}$ is a subcategory of $\cosubproj{A}$.
        \end{enumerate}
    
    \end{proof}
    
    \begin{eg} \label{eg:relativelyinjectivefrobeniusandalmostexcellentextensions}
        Lemma~\ref{lem:ring_ext_relatively_proj_inj_generation} applies to both Frobenius extensions and almost excellent extensions.
    
        \begin{itemize}
            \item \underline{Frobenius extensions.}
        
            Let $A$ and $B$ be rings such that $A$ is a Frobenius extension of $B$. Let $M_A$ be an $A$-module. Then $M_A$ is relatively $B$-injective if and only if $M_A$ is relatively $B$-projective, \cite[Proposition 4.1]{Kadison1999}.
            
            \item \underline{Almost excellent extensions}
            
            Recall that, if $A$ is an almost excellent extension of $B$, then $A$ is right $B$-projective. Hence every $A$-module is both relatively $B$-injective and relatively $B$-projective, \cite[Lemma 1.1]{Xue1996}.
        \end{itemize}
    \end{eg}

\section{Recollements of derived module categories} \label{section:recollements}

Recollements of triangulated categories were first introduced by Be\u{\i}linson, Bernstein and Deligne \cite{Beuilinson1982} to study stratifications of derived categories of sheaves. Recall the definition of a recollement of derived module categories.

\begin{defn}[Recollement]
	Let $A$, $B$ and $C$ be rings. A recollement is a diagram of six triangle functors as in Figure~\ref{fig:recollement} such that the following hold:
	\begin{enumerate}[label=(\roman*)]
		\item The composition $j^* \circ i_*=0$.
		\item All of the pairs $(i^*,i_*)$, $(i_*,i^!)$, $(j_!,j^*)$ and $(j^*,j_*)$ are adjoint pairs of functors.
		\item The functors $i_*$, $j_!$ and $j_*$ are fully faithful.
		\item For all $X \in \der{A}$ there exist triangles:
		\begin{align*}
		    \tri{j_!j^*X}{X}{i_*i^*X}{}{}{}
		    \\
		    \tri{i_*i^!X}{X}{j_*j^*X}{}{}{}
		\end{align*}
	\end{enumerate}
\end{defn}

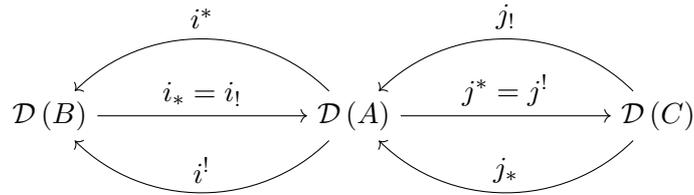
\begin{figure}[H]
\begin{center}
	\begin{tikzpicture}
		\node(a) at (-4.5,0) {};
		\node (t') at (-4,0) {$\der{B}$};
		\node (t) at (0,0) {$\der{A}$};
		\node (t'') at (4,0) {$\der{C}$};
		
		\draw[->] (t') -- node[midway,above] {$i_* = i_!$} (t);
		\draw[->] (t) -- node[midway,above] {$j^*=j^!$} (t'');
		\draw[->] (t) to [bend right = 45] node[midway,above] {$i^*$} (t');
		\draw[->] (t) to [bend left = 45] node[midway,above] {$i^!$} (t');
		\draw[->] (t'') to [bend right = 45] node[midway,above] {$j_!$} (t);
		\draw[->] (t'') to [bend left = 45] node[midway,above] {$j_*$} (t);		
	\end{tikzpicture}	
\end{center}
\caption{Recollement of derived categories $(R)$}
\label{fig:recollement}
\end{figure}

We denote a recollement of the form in Figure~\ref{fig:recollement} as $(R) = (\der{B},\der{A},\der{C})$. If a recollement $(R)$ exists then the properties of $A$, $B$ and $C$ are often related. This allows one to prove properties about $A$ using the usually simpler $B$ and $C$. Such a method has been exploited by Happel \cite[Theorem 2]{Happel1993} and Chen and Xi \cite[Theorem 1.1]{Chen2017} to prove various statements about the finitistic dimension of rings that appear in recollements. These results apply to recollements $(R)$ which restrict to recollements on derived categories with various boundedness conditions. In this section we say a recollement $(R)$ restricts to a recollement $(R^*)$ for $* \in \{-,+,b\}$ if the six functors of $(R)$ restrict to functors on the essential image of $\derstarnone{}$ in $\cat{D}(\text{Mod})$. Note that such a restriction is not always possible, however in \cite[Section 4]{Huegel2017} there are necessary and sufficient conditions for $(R)$ to restrict to a recollement $(R^-)$ or $(R^b)$. In Proposition~\ref{prop:recollementbdedbelowconditions} we provide analogous conditions for $(R)$ to restrict to a recollement $(R^+)$.

\begin{eg} \label{eg:triangularmatrixring}
    One example of a recollement of unbounded derived module categories can be defined using triangular matrix rings. Let $B$ and $C$ be rings and $_CM_B$ a $(C,B)$-bimodule. Then the triangular matrix ring is defined as
    \begin{equation*}
        A \coloneqq \twobytwo{C}{_CM_B}{0}{B}.
    \end{equation*}
    In this situation $A$, $B$ and $C$ define a recollement $(R)$. The functors of $(R)$ are defined using idempotents of $A$. Let 
    \begin{equation*}
        e_1 \coloneqq \twobytwo{1}{0}{0}{0}, \hspace{2mm} e_2 \coloneqq \twobytwo{0}{0}{0}{1}.
    \end{equation*}
    Then the recollement is defined by $j_! = \ltens{-}{C}{e_1A}$ and $i_* = \ltens{-}{B}{e_2A}$.
\end{eg}

This section includes results about the dependence of $A$, $B$ and $C$ on each other with regards to `injectives generate' and `projectives cogenerate' statements. Theorem~\ref{thm:recollementsbcimplya} is a summary of the results in this section which use properties of $B$ and $C$ to prove generation statements about $A$.

\begin{thm} \label{thm:recollementsbcimplya}
    Let $(R)$ be a recollement. 
    \begin{enumerate}[label = (\roman*)]
        \item Suppose that injectives generate for both $B$ and $C$. If one of the following conditions holds then injectives generate for $A$.
        \begin{enumerate}[label = (\alph*)]
            \item The recollement $(R)$ is in a ladder of recollements with height greater than or equal to $2$. [Proposition~\ref{prop:recollementsextonedown}]
            \item The recollement $(R)$ restricts to a bounded below recollement $(R^+)$.  [Proposition~\ref{prop:recollementbdedbelow}]
            \item The recollement $(R)$ restricts to a bounded above recollement $(R^-)$ and $A$ is a finite dimensional algebra over a field. [Proposition~\ref{prop:recollementsbdedabove}]
        \end{enumerate}
    
        \item Suppose that projectives cogenerate for both $B$ and $C$. If one of the following conditions holds then projectives cogenerate for $A$.
        \begin{enumerate}[label = (\alph*)]
            \item The recollement $(R)$ is in a ladder of recollements with height greater than or equal to $2$. [Proposition~\ref{prop:recollementsextonedown}]
            \item The recollement $(R)$ restricts to a bounded above recollement $(R^-)$. [Proposition~\ref{prop:recollementsbdedabove}]
            \item The recollement $(R)$ restricts to a bounded below recollement $(R^+)$ and $A$ is a finite dimensional algebra over a field. [Proposition~\ref{prop:recollementbdedbelow}]
        \end{enumerate}
    \end{enumerate}
\end{thm}

To prove Theorem~\ref{thm:recollementsbcimplya} we require some technical results which we state and prove now. We prove these results by using the fact there are four pairs of adjoint functors in a recollement. The results in Table~\ref{table:recollementsfunctorprop} follow from the definition of a recollement, standard properties of adjoint functors, and the results of Section~\ref{section:localisingsubcategory}. Recall that for brevity, we shall often abuse terminology by writing, for example, that a functor preserves bounded complexes of projectives when we mean that it preserves the property of being quasi-isomorphic to a bounded complex of projectives.

\setlength\tabcolsep{5pt}

\begin{table}[H]
\begin{center}
	\begin{tabular}{|c|c|}
	    \hline
	    \textbf{Property preserved} & \textbf{Functors with this property}
	    \\
	    \hline
    	Products & $i_*$, $i^!$, $j^*$, $j_*$.
		\\
		Coproducts & $i^*,i_*,j_!,j^*$.
		\\
		Compact objects & $i^*,j_!$.
		\\
		Complexes bounded in cohomology & $i_*,j^*$.
		\\
		Complexes bounded above in cohomology & $i^*,i_*,j_!,j^*$.
		\\
		Complexes bounded below in cohomology & $i_*,i^!,j^*,j_*$.
		\\
		Bounded complexes of projectives & $i^*,j_!$.
		\\
		Bounded complexes of injectives & $i^!,j_*$. 
		\\
		\hline
	\end{tabular}
\end{center}
\caption{Properties of the triangle functors in a recollement}
\label{table:recollementsfunctorprop}
\end{table}

\begin{lemma} \label{lemma:recollementj*}
    Let $(R)$ be a recollement.
    \begin{enumerate}[label = (\roman*)]
        \item If $j^*$ preserves bounded complexes of injectives and injectives generate for $A$ then injectives generate for $C$.
        \item If $j^*$ preserves bounded complexes of projectives and projectives cogenerate for $A$ then projectives cogenerate for $C$.
    \end{enumerate}
\end{lemma}

\begin{proof}
    \begin{enumerate}[label = (\roman*)]
        \item Suppose that injectives generate for $A$. Since $j^*$ preserves bounded complexes of injectives and set indexed coproducts, $\im{j^*}$ is a subcategory of $\lsubinj{C}$ by Proposition~\ref{prop:imageoffunctor}. Furthermore, $j^*$ is essentially surjective as it is right adjoint to $j_!$, which is fully faithful. Hence injectives generate for $C$. 
        
        \item The proof of the second statement is similar.
    \end{enumerate}
\end{proof}

\begin{prop} \label{prop:recollementimi*}
    Let $(R)$ be a recollement.
    \begin{enumerate}[label = (\roman*)]
        \item If $\im{i_*}$ is a subcategory of $ \lsubinj{A}$ and injectives generate for $C$ then injectives generate for $A$.
        \item If $\im{i_*}$ is a subcategory of $\cosubproj{A}$ and projectives cogenerate for $C$ then projectives cogenerate for $A$.
    \end{enumerate}
\end{prop}

\begin{proof}
    \begin{enumerate}[label = (\roman*)]
        \item Suppose that $\im{i_*}$ is a subcategory of $\lsubinj{A}$. Let $I \in \der{C}$ be a bounded complex of injectives. Consider the triangle,
        \begin{equation} \label{tri:recollementsimageofi*}
            \tri{j_!j^*(j_*(I))}{j_*(I)}{i_*i^*(j_*(I))}{}{}{}.
        \end{equation}
        Since $j_*$ preserves bounded complexes of injectives, $j_*(I)$ is in $\lsubinj{A}$. Hence triangle~\ref{tri:recollementsimageofi*} implies that $j_!j^*(j_*(I))$ is in $\lsubinj{A}$. Recall that $j_*$ is fully faithful so $j_!j^*j_*(I) \cong j_!(I)$. Thus $j_!$ maps bounded complexes of injectives to $\lsubinj{A}$. 
        
        Suppose that injectives generate for $C$. Then, by Proposition~\ref{prop:imageoffunctor}, $\im{j_!}$ is a subcategory of $\lsubinj{A}$. Since $\im{i_*}$ and $\im{j_!}$ are both subcategories of $\lsubinj{A}$, for all $X \in \der{A}$ both $i_*i^*(X)$ and $j_!j^*(X)$ are in $\lsubinj{A}$. Hence $X$ is in $\lsubinj{A}$ by the triangle
        \begin{equation*}
            \tri{j_!j^*(X)}{X}{i_*i^*(X)}{}{}{}.
        \end{equation*}
        Thus injectives generate for $A$.
        
        \item The second result follows similarly.
    \end{enumerate}
\end{proof}

\begin{prop} \label{prop:recollementi*}
    Let $(R)$ be a recollement. 
    \begin{enumerate}[label = (\roman*)]
        \item Suppose that $i_*$ preserves bounded complexes of injectives.
        \begin{enumerate}
            \item If injectives generate for both $B$ and $C$ then injectives generate for $A$.
            \item If injectives generate for $A$ then injectives generate for $C$.
        \end{enumerate}
        \item Suppose that $i_*$ preserves bounded complexes of projectives.
        \begin{enumerate}
            \item If projectives cogenerate for both $B$ and $C$ then projectives cogenerate for $A$.
            \item If projectives cogenerate for $A$ then projectives cogenerate for $C$.
        \end{enumerate}
    \end{enumerate}
\end{prop}

\begin{proof}
    \begin{enumerate}[label = (\roman*)]
        \item We prove the first two statements as the other two follow similarly.
        \begin{enumerate}[label = (\alph*)]
            \item Suppose that injectives generate for $B$. Since $i_*$ preserves bounded complexes of injectives and set indexed coproducts, by Proposition~\ref{prop:imageoffunctor}, $\im{i_*}$ is a subcategory of $\lsubinj{A}$. Thus, if injectives generate for $C$, then injectives generate for $A$ by Proposition~\ref{prop:recollementimi*}.
            
            \item We claim that $j^*$ also preserves bounded complexes of injectives. Since $j^*$ preserves complexes bounded in cohomology, by Lemma~\ref{lemma:adjointfunctorspreserve}, $j_!$ preserves complexes bounded above in cohomology and $j_*$ preserves complexes bounded below in cohomology. Furthermore, since $i_*$ preserves bounded complexes of injectives, by Lemma~\ref{lemma:adjointfunctorspreserve}, $i^*$ preserves complexes bounded below in cohomology.
            Let $X \in \der{C}$ be bounded below in cohomology and consider the triangle
            \begin{equation*}
                \begin{tikzcd}[ampersand replacement=\&]
                    j_!j^*(j_*(X)) \arrow[r,""] \arrow[d, "\cong"] \&
                      j_*(X) \arrow[r,""] \arrow[d, "="] \&
                        i_*i^*(j_*(X)) \arrow[r,""] \arrow[d, "="] \&
                          j_!j^*(j_*(X))[1] \arrow[d, "\cong"] 
                    \\
                    j_!(X) \arrow[r,""] \&
                      j_*(X) \arrow[r,""] \&
                        i_*i^*(j_*(X)) \arrow[r,""] \&
                          j_!(X)[1].
                \end{tikzcd}
            \end{equation*}
            Since $i_*$, $i^*$ and $j_*$ all preserve complexes bounded below in cohomology, by the triangle, $j_!$ also preserves complexes bounded below in cohomology. Hence $j_!$ preserves both complexes bounded above and bounded below in cohomology. Thus $j_!$ preserves complexes bounded in cohomology and $j^*$ preserves bounded complexes of injectives, by Lemma~\ref{lemma:adjointfunctorspreserve}. Hence the statement follows immediately from Lemma~\ref{lemma:recollementj*}.
        \end{enumerate}
        \item This follows similarly to above.
    \end{enumerate}

\end{proof}

\begin{lemma} \label{lem:recollementsi*andi!}
    Let $(R)$ be a recollement.
    \begin{enumerate}[label= (\roman*)]
        \item Suppose that injectives generate for $A$. Then injectives generate for $B$ if one of the following two conditions holds:
        \begin{enumerate}
            \item $i^!$ preserves set indexed coproducts.
            \item $i^*(\Inj{A})$ is a subcategory of $\lsubinj{B}$.
        \end{enumerate}
        
        \item Suppose that projectives cogenerate for $A$. Then projectives cogenerate for $B$ if one of the following two conditions holds:
        \begin{enumerate}
            \item $i^*$ preserves set indexed products.
            \item $i^!(\Proj{A})$ is a subcategory of $\cosubproj{B}$.
        \end{enumerate}
    \end{enumerate}
\end{lemma}

\begin{proof}
    \begin{enumerate}[label = (\roman*)]
        \item Since $i_*$ is fully faithful both $i^*$ and $i^!$ are essentially surjective. Hence, if $\im{i^*}$ or $\im{i^!}$ is a subcategory of $\lsubinj{B}$, then injectives generate for $B$. The two statements are sufficient conditions for this to happen using Proposition~\ref{prop:imageoffunctor}.
        
        \item The proof is similar for the second statement.
    \end{enumerate}
\end{proof}

\subsection{Ladders of recollements}

A ladder of recollements, \cite[Section 3]{Huegel2017}, is a collection of finitely or infinitely many rows of triangle functors between $\der{A}$, $\der{B}$ and $\der{C}$, of the form given in Figure~\ref{fig:recollementladder}, such that any three consecutive rows form a recollement. The height of a ladder is the number of distinct recollements it contains.

\begin{figure}[H]
\begin{center}
	\begin{tikzpicture}
		\node(a) at (-4,0) {};
		\node (t') at (-3.5,0) {$\der{B}$};
		\node (t) at (0,0) {$\der{A}$};
		\node (t'') at (3.5,0) {$\der{C}$};
		
		\draw[->] (- 2.5,0) -- node[midway,above] {$i_n$} (-1,0);
		\draw[->] (1,0) -- node[midway,above] {$j_n$} ( 2.5,0);
		\draw[<-] (- 2.5,0.5) to [bend left = 15] node[midway,above] {$j_{n-1}$} (-1,0.5);
		\draw[<-] (1,0.5) to [bend left = 15] node[midway,above] {$i_{n-1}$} ( 2.5,0.5);
		\draw[<-] (- 2.5,-0.5) to [bend right = 15] node[midway,above] {$j_{n+1}$} (-1,-0.5);
		\draw[<-] (1,-0.5) to [bend right = 15] node[midway,above] {$i_{n+1}$} ( 2.5,-0.5);	
		\draw[->] (- 2.5,1) to [bend left = 40] node[midway,above] {$i_{n-2}$} (-1,1);
		\draw[->] (- 2.5,-1) to [bend right = 40] node[midway,above] {$i_{n+2}$} (-1,-1);
		\draw[->] (1,-1) to [bend right = 40] node[midway,above] {$j_{n+2}$} ( 2.5,-1);
		\draw[->] (1,1) to [bend left = 40] node[midway,above] {$j_{n-2}$} ( 2.5,1);
		
		\draw[dotted,thick] (-2,2) -- (-2,2.35);
		\draw[dotted,thick] (-2,-2) -- (-2,-2.35);
		\draw[dotted,thick] (2,2) -- (2,2.35);
		\draw[dotted,thick] (2,-2) -- (2,-2.35);
		\end{tikzpicture}	
\end{center}
\caption{Ladder of recollements}
\label{fig:recollementladder}
\end{figure}
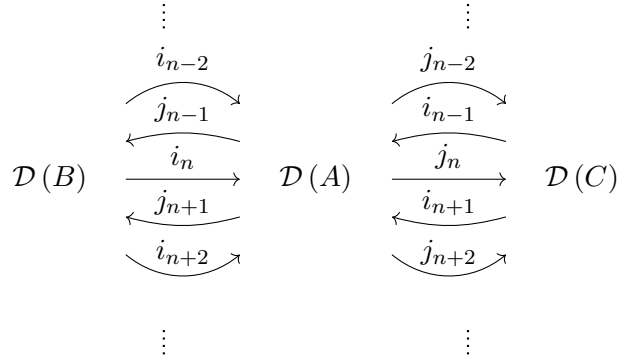

\begin{prop} \cite[Proposition 3.2]{Huegel2017} \label{Huegel2017:recollementclassificationext}
    Let $(R)$ be a recollement.
    \begin{enumerate}[label = (\roman*)]
        \item The recollement $(R)$ can be extended one step downwards if and only if $j_*$ (equivalently $i^!$) has a right adjoint. This occurs exactly when $j^*$ (equivalently $i_*$) preserves compact objects.
        \item The recollement $(R)$ can be extended one step upwards if and only if $j_!$ (equivalently $i^*$) has a left adjoint. If $A$ is a finite dimensional algebra over a field this occurs exactly when $j_!$ (equivalently $i^*$) preserves bounded complexes of finitely generated modules.
    \end{enumerate}
\end{prop}

If the recollement $(R)$ can be extended one step downwards then we have a recollement $(R_\downarrow)$ as in Figure~\ref{fig:recollementextdown}.

\begin{figure}[H]
\begin{center}
	\begin{tikzpicture}
		\node(a) at (-4.5,0) {};
		\node (t') at (-4,0) {$\der{B}$};
		\node (t) at (0,0) {$\der{A}$};
		\node (t'') at (4,0) {$\der{C}$};
		\node(l) at (4.5,1) {Left};
		\node(r) at (4.5,-1) {Right};
		
		\draw[->] (t') to [bend left = 15]  node[midway,above] {$i_* = i_!$} (t);
		\draw[->] (t) to [bend left = 15] node[midway,above] {$j^*=j^!$} (t'');
		\draw[->] (t) to [bend right = 60] node[midway,above] {$i^*$} (t');
		\draw[->] (t) to [bend left = 15] node[midway,below] {$i^!$} (t');
		\draw[->] (t'') to [bend right = 60] node[midway,above] {$j_!$} (t);
		\draw[->] (t'') to [bend left = 15] node[midway,below] {$j_*$} (t);
		\draw[->] (t') to [bend right = 60] node[midway,below] {{$i_\downarrow$}} (t);
		\draw[->] (t) to [bend right = 60] node[midway,below] {{$j^\downarrow$}} (t'');
	\end{tikzpicture}	
\end{center}
\caption{Recollement of derived categories extended one step downwards $(R_\downarrow)$}
\label{fig:recollementextdown}
\end{figure}
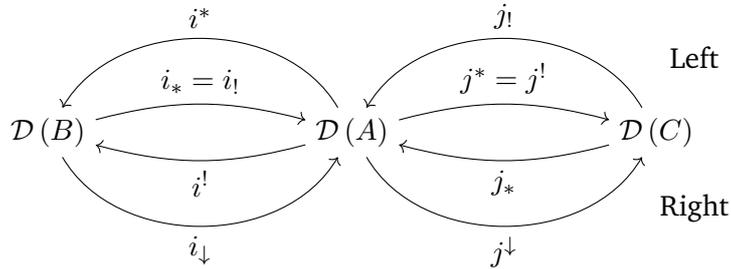

\begin{eg} 
    As seen in Example~\ref{eg:triangularmatrixring} a triangular matrix ring defines a recollement $(R)$. Moreover, this recollement extends one step downwards, \cite[Example 3.4]{Huegel2017}. Recall that $i_* \coloneqq \ltens{-}{B}{e_2A}$ where $e_2$ is an idempotent of  $A$. In particular, note that $e_2A_A$ is a finitely generated projective $A$-module so $i_*$ preserves compact objects. Thus we can apply Proposition~\ref{Huegel2017:recollementclassificationext} to show that $(R)$ extends one step downwards. 
\end{eg}
	
\begin{prop} \label{prop:recollementsextonedown}
Let $(R)$ be the top recollement in a ladder of height $2$.
\begin{enumerate}[label = (\roman*)]
    \item If injectives generate for both $B$ and $C$ then injectives generate for $A$.
    \item If injectives generate for $A$ then injectives generate for $B$.
    \item If projectives cogenerate for both $B$ and $C$ then projectives cogenerate for $A$.
    \item If projectives cogenerate for $A$ then projectives cogenerate for $C$.
\end{enumerate}
\end{prop}

\begin{proof}
    The bottom recollement of the ladder is a recollement as in $(R)$ but with the positions of $B$ and $C$ swapped. Hence in this bottom recollement $j_*$ acts as $i_*$ does in the recollement $(R)$. Moreover, $j_*$ preserves bounded complexes of injectives. Thus we apply Proposition~\ref{prop:recollementi*} to prove (i) and (ii).
    
    Moreover, since $i^!$ has a right adjoint, $i^!$ preserves set indexed coproducts. Hence $i_*$ preserves bounded complexes of projectives by Lemma~\ref{lemma:adjointfunctorspreserve}. Consequently, Proposition~\ref{prop:recollementi*} proves (iii) and (iv).
\end{proof}

\begin{eg}
    By Proposition~\ref{prop:recollementsextonedown} it follows immediately that for any triangular matrix ring 
    \begin{equation*}
        A = \twobytwo{C}{_CM_B}{0}{B},
    \end{equation*}
    if injectives generate for $B$ and $C$ then injectives generate for $A$. Moreover, if injectives generate for $A$ then injectives generate for $B$. 
    
    Note that this is the `injectives generate' version of the well known result proved by Fossum, Griffith and Reiten \cite[Corollary 4.21]{FossumGriffithReiten1975} about the finitistic dimensions of triangular matrix rings.
\end{eg}

\begin{cor}
    Let $(R)$ be a recollement in a ladder of height $ \geq 3$. 
    \begin{enumerate}[label=(\roman*)]
        \item Then injectives generate for $A$ if and only if injectives generate for both $B$ and $C$.
        \item Then projectives cogenerate for $A$ if and only if projectives cogenerate for both $B$ and $C$.
    \end{enumerate}
\end{cor}

\subsection{Bounded above recollements}

In this section we consider the case of a recollement which restricts to a bounded above recollement. In particular, we use a characterisation by \cite{Huegel2017}.

\begin{prop}\cite[Proposition 4.11]{Huegel2017} \label{prop:recollementsbdedaboveconditions}
    Let $(R)$ be a recollement. Then the following are equivalent:
    \begin{enumerate}[label = (\roman*)]
        \item The recollement $(R)$ restricts to a bounded above recollement $(R^-)$,
        \item The functor $i_*$ preserves bounded complexes of projectives.
    \end{enumerate}
    If $A$ is a finite dimensional algebra over a field then both conditions are equivalent to:
    \begin{enumerate}[label = (\roman*),resume]
        \item The functor $i_*$ preserves compact objects.
    \end{enumerate}
\end{prop}

Note that if $i_*(B)$ is compact then the recollement $(R)$ also extends one step downwards by Proposition~\ref{Huegel2017:recollementclassificationext} \cite[Proposition 3.2]{Huegel2017}.

\begin{prop} \label{prop:recollementsbdedabove}
Let $(R)$ be a recollement that restricts to a bounded above recollement $(R^-)$.
\begin{enumerate}[label = (\roman*)]
    \item If projectives cogenerate for $B$ and $C$ then projectives cogenerate for $A$. \item If projectives cogenerate for $A$ then projectives cogenerate for $C$.
\end{enumerate}
Moreover, suppose that $A$ is a finite dimensional algebra over a field.
\begin{enumerate}[label = (\roman*),resume]
    \item If injectives generate for $B$ and $C$ then injectives generate for $A$.
    \item If injectives generate for $A$ then injectives generate for $B$.
\end{enumerate}

\end{prop}

\begin{proof}
Since $(R)$ restricts to a recollement of bounded above derived categories, $i_*$ preserves bounded complexes of projectives by Proposition \ref{prop:recollementsbdedaboveconditions} \cite[Proposition 4.11]{Huegel2017}. Hence we apply Proposition~\ref{prop:recollementi*} to get (i) and (ii). Furthermore, if $A$ is a finite dimensional algebra over a field then $i_*$ preserves compact objects. Then the recollement also extends one step downwards by Proposition~\ref{Huegel2017:recollementclassificationext} so (iii) and (iv) follow from Proposition~\ref{prop:recollementsextonedown}.

\end{proof}

\subsection{Bounded below recollements}

Now we consider recollements $(R)$ that restrict to bounded below recollements $(R^+)$. First we prove an analogous statement to Proposition~\ref{prop:recollementsbdedaboveconditions} about the conditions under which a recollement $(R)$ restricts to a recollement $(R^+)$.

\begin{prop} \label{prop:recollementbdedbelowconditions}
    Let $(R)$ be a recollement. Then the following are equivalent:
    \begin{enumerate}[label = (\roman*)]
        \item The recollement $(R)$ restricts to a bounded below recollement $(R^+)$,
        \item The functor $i_*$ preserves bounded complexes of injectives.
    \end{enumerate}
    If $A$ is a finite dimensional algebra over a field then both conditions are equivalent to:
    \begin{enumerate}[label = (\roman*),resume]
        \item The functor $j_!$ preserves bounded complexes of finitely generated modules.
    \end{enumerate}
\end{prop}

\begin{proof}

    First we prove that (ii) implies (i). Suppose that $i_*$ preserves bounded complexes of injectives. Then by the proof of Proposition~\ref{prop:recollementi*} all six functors preserve complexes bounded below in cohomology. Hence the recollement $(R)$ restricts to a bounded below recollement $(R^+)$. 
   
    For the converse statement, suppose that $(R)$ restricts to a bounded below recollement $(R^+)$, that is, all six functors preserve complexes bounded below in cohomology. Since $i_*$ preserves complexes bounded in cohomology, by Lemma~\ref{lemma:adjointfunctorspreserve}, $i^*$ preserves complexes bounded above in cohomology. Hence $i^*$ preserves both complexes bounded above and bounded below in cohomology. Thus $i^*$ preserves complexes bounded in cohomology and by Lemma~\ref{lemma:adjointfunctorspreserve}, $i_*$ preserves bounded complexes of injectives.
    
    Note that (iii) implies (ii) by Proposition~\ref{Huegel2017:recollementclassificationext} and Lemma~\ref{lemma:adjointfunctorspreserve}.
    
    Finally, we show that (i) implies (iii). Suppose that $A$ is a finite dimensional algebra over a field. Let $X \in \der{C}$ be a bounded complex of finitely generated modules. Since $A$ is a finite dimensional algebra over a field, $j_!(X)$ is a bounded above complex of finitely generated modules \cite[Lemma 2.10 (b)]{Huegel2017}. Suppose that $(R)$ restricts to a bounded below recollement $(R^+)$. Then $j_!(X)$ is bounded below in cohomology, so we can truncate $j_!(X)$ from below and $j_!(X)$ is quasi-isomorphic to a bounded complex of finitely generated modules.
\end{proof}

We can use these results to get an analogous statement to Proposition~\ref{prop:recollementsbdedabove} about bounded below recollements.

\begin{prop} \label{prop:recollementbdedbelow}
    Let $(R)$ be a recollement that restricts to a bounded below recollement $(R^+)$.
    \begin{enumerate}[label = (\roman*)]
        \item If injectives generate for both $B$ and $C$ then injectives generate for $A$.
        \item If injectives generate for $A$ then injectives generate for $C$.
    \end{enumerate}
    Moreover, suppose that $A$ is a finite dimensional algebra over a field.
    \begin{enumerate}[label = (\roman*),resume]
            \item If projectives cogenerate for both $B$ and $C$ then projectives cogenerate for $A$.
            \item If projectives cogenerate for $A$ then projectives cogenerate for $B$.
    \end{enumerate}
\end{prop}

\begin{proof}
    The proof is dual to the proof of Proposition~\ref{prop:recollementsbdedabove}.
\end{proof}

\subsection{Bounded recollements}

Finally we consider the case of a recollement $(R)$ which restricts to a bounded recollement $(R^b)$.

\begin{prop} \label{prop:recollementsbded}
Let $(R)$ be a recollement that restricts to a bounded recollement $(R^b)$.
\begin{enumerate}[label = (\roman*)]
     \item If injectives generate for both $B$ and $C$ then injectives generate for $A$.
     \item If injectives generate for $A$ then injectives generate for $C$.
     \item If projectives cogenerate for both $B$ and $C$ then projectives cogenerate for $A$.
     \item If projectives cogenerate for $A$ then projectives cogenerate for $C$.
\end{enumerate}
Moreover, suppose that $A$ is a finite dimensional algebra over a field.
\begin{enumerate}[label = (\roman*), resume]
    \item Injectives generate for $A$ if and only if injectives generate for both $B$ and $C$.
    \item Projectives cogenerate for $A$ if and only if projectives cogenerate for both $B$ and $C$.
\end{enumerate}
\end{prop}

\begin{proof}
    Since $(R^b)$ is a recollement of bounded derived categories both $i^*$ and $i^!$ preserve complexes bounded in cohomology. Hence, by Lemma~\ref{lemma:adjointfunctorspreserve}, $i_*$ preserves both bounded complexes of injectives and bounded complexes of projectives. So $(R)$ restricts to a bounded above recollement by Proposition~\ref{prop:recollementsbdedaboveconditions} and a bounded below recollement  by Proposition~\ref{prop:recollementbdedbelowconditions}. Thus the results follow immediately from Proposition~\ref{prop:recollementsbdedabove} and Proposition~\ref{prop:recollementbdedbelow}.
\end{proof}

\section{Recollements of module categories} \label{section:dgrecollements}

Recollements can also be defined for abelian categories, in particular, for module categories, see for example \cite[Definition 2.6]{PsaroudakisVitoria2014}.

\begin{defn}[Recollement of module categories]
    Let $A$, $B$ and $C$ be rings. A recollement of module categories is a diagram of additive functors as in Figure~\ref{fig:recollementmodulecategories} such that the following hold:
    \begin{enumerate}[label = (\roman*)]
        \item Each of the pairs $(q,i)$, $(i,p)$, $(l,e)$ and $(e,r)$ is an adjoint pair of functors.
        \item The functors $i$, $l$ and $r$ are fully faithful.
        \item The composition $e \circ i$ is zero.
    \end{enumerate}
\end{defn}

\begin{figure}[htbp]
\begin{center}
	\begin{tikzpicture}
		\node(a) at (-5,0) {};
		\node (t') at (-4.5,0) {$\Mod{B}$};
		\node (t) at (0,0) {$\Mod{A}$};
		\node (t'') at (4.5,0) {$\Mod{C}$};
		
		\draw[->] (t') -- node[midway,above] {$i$} (t);
		\draw[->] (t) -- node[midway,above] {$e$} (t'');
		\draw[->] (t) to [bend right = 45] node[midway,above] {$q$} (t');
		\draw[->] (t) to [bend left = 45] node[midway,below] {$p$} (t');
		\draw[->] (t'') to [bend right = 45] node[midway,above] {$l$} (t);
		\draw[->] (t'') to [bend left = 45] node[midway,below] {$r$} (t);		
	\end{tikzpicture}	
\end{center}
\vspace{-0.325cm}
\caption{Recollement of module categories }
\label{fig:recollementmodulecategories}
\end{figure}
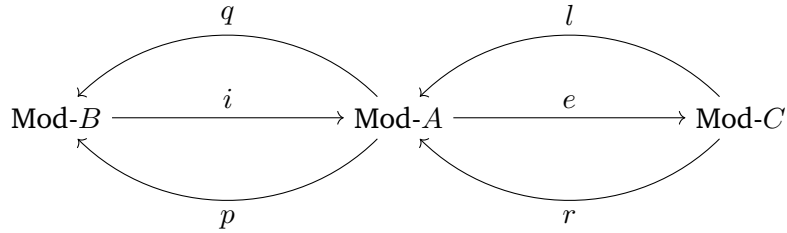

We will denote a recollement of module categories as in Figure~\ref{fig:recollementmodulecategories} by $(\Mod{B}, \Mod{A}, \Mod{C})$. Recollements of module categories are prevalent in representation theory as given a ring $A$ and an idempotent $e \in A$ there exists a recollement of module categories
\begin{equation*}
    (\Mod{(\quotient{A}{AeA}{})},\Mod{A},\Mod{eAe}),
\end{equation*}
with functors given by, 

\begin{table}[H]
        \begin{center}
            \begin{tabular}{cc}
                 $q \coloneqq \tens{-}{A}{\quotient{A}{AeA}{}}$, &  $l \coloneqq \tens{-}{eAe}{eA}$,
                 \vspace{1mm}
                 \\
                 $i \coloneqq \Homgroup{\quotient{A}{AeA}{}}{\quotient{A}{AeA}{}}{-} \cong \tens{-}{\quotient{A}{AeA}{}}{\quotient{A}{AeA}{}}$, & $e \coloneqq \Homgroup{A}{eA}{-} \cong \tens{-}{A}{Ae}$,
                 \vspace{1mm}
                 \\
                 $p \coloneqq \Homgroup{A}{\quotient{A}{AeA}{}}{-}$, & $r \coloneqq \Homgroup{eAe}{Ae}{-}$.
            \end{tabular}
        \end{center}
    \end{table}

The recollement of module categories $(\Mod{(\quotient{A}{AeA}{})},\Mod{A},\Mod{eAe})$ lifts to a recollement of the corresponding derived module categories if $\function{\pi}{A}{\quotient{A}{AeA}{}}$ is a homological epimorphism \cite{Cline1996}, \cite[Subsection 1.7]{AngeleriHuegelKoenigLiu2011}, \cite[Proposition 2.1]{Huegel2017-dgalgebra}. When this condition is not satisfied the recollement of module categories lifts to a recollement of derived module categories of dg algebras \cite[Remark p.~55]{Huegel2017-dgalgebra}. In this case the lifted recollement has middle ring $A$ and right hand side ring $eAe$ with $j_! = \textbf{L}l$, $j^* = e$ and $j_* = \textbf{R}r$. However, the left hand side ring is given by a dg algebra $B$, see Figure~\ref{fig:recollementlifteddgalgebra}.

\begin{figure}[htbp]
\begin{center}
	\begin{tikzpicture}
		\node(a) at (-5,0) {};
		\node (t') at (-4.5,0) {$\der{B}$};
		\node (t) at (0,0) {$\der{A}$};
		\node (t'') at (4.5,0) {$\der{eAe}$};
		
		\draw[->] (t') -- node[midway,above] {$i_*$} (t);
		\draw[->] (t) -- node[midway,above] {$\tens{-}{A}{Ae}$} (t'');
		\draw[->] (t) to [bend right = 45] node[midway,above] {$i^*$} (t');
		\draw[->] (t) to [bend left = 45] node[midway,below] {$i^!$} (t');
		\draw[->] (t'') to [bend right = 45] node[midway,above] {$\ltens{-}{eAe}{eA}$} (t);
		\draw[->] (t'') to [bend left = 45] node[midway,below] {$\rHom{eAe}{Ae}{-}$} (t);		
	\end{tikzpicture}	
\end{center}
\vspace{-0.325cm}
\caption{Recollement of dg algebras induced from a recollement of module categories }
\label{fig:recollementlifteddgalgebra}
\end{figure}

Throughout the rest of this section we focus on proving generation statements relating the three rings $A$, $eAe$ and $\quotient{A}{AeA}{}$. To do this we aim to apply the results of Section~\ref{section:recollements} to the induced recollement of derived categories in Figure~\ref{fig:recollementlifteddgalgebra}.

\begin{lemma} \label{lem:homologyinsubcategories}
    Let $A$ be a ring. Let $X \in \der{A}$ and let $\cat{T}$ be a triangulated subcategory of $\der{A}$. Suppose that all of the cohomology modules of $X$ are in $\cat{T}$. If $X \in \der{A}$ is bounded in cohomology then $X$ is in $\cat{T}$.
\end{lemma}

\begin{proof}
    A complex bounded in cohomology is in the triangulated subcategory of the derived category generated by its cohomology modules.
\end{proof}

We restrict to the case when the dg algebra of the recollement of derived categories induced by the recollement $(\Mod{(\quotient{A}{AeA}{})},\Mod{A},\Mod{eAe})$ is bounded in cohomology. This property is satisfied when, for example, $Ae$ has finite flat dimension as a right $eAe$-module or $eA$ has finite flat dimension as a left $eAe$-module.

\begin{lemma} \label{lemma:imi*isgeneratedbyA/AeA}
    Let $A$ be a ring and $e \in A$ an idempotent. Consider the functor
    \begin{equation*}
        \function{j^* \coloneqq \tens{-}{A}{Ae}}{\der{A}}{\der{eAe}}.
    \end{equation*}
    Suppose that $\ltens{Ae}{eAe}{eA}$ is bounded in cohomology.
    \begin{enumerate}[label = (\roman*)]
        \item Then $\kernel{\tens{-}{A}{Ae}} = \lsub{A}{\res{A}{\quotient{A}{AeA}{}}{\quotient{A}{AeA}{}}}$. 
        \item Then $\kernel{\tens{-}{A}{Ae}} = \cosub{A}{\res{A}{\quotient{A}{AeA}{}}{\cogenerator{(\quotient{A}{AeA}{})}}}$.
    \end{enumerate}

\end{lemma}

\begin{proof}
    Denote the restriction functor $\res{A}{\quotient{A}{AeA}{}}{}$ by $\res{}{}{}$. Note that there exists a recollement of module categories $(\Mod{(\quotient{A}{AeA}{})},\Mod{A},\Mod{eAe})$ which lifts to a recollement of derived module categories $(R) = (\der{B},\der{A},\der{eAe})$ where $B$ is a dg algebra. In the recollement $(R)$ the functor $\function{j^*}{\der{A}}{\der{eAe}}$ is equal to $\tens{-}{A}{Ae}$. Moreover, since $(R)$ is a recollement $\kernel{j^*} = \im{i_*}$.

    \begin{enumerate}[label = (\roman*)]
        \item Since $j^*(\res{}{}{\quotient{A}{AeA}{}}) = 0$, $i_*$ is fully faithful and $i_*$ preserves set indexed coproducts, $\lsub{A}{\res{}{}{\quotient{A}{AeA}{}}}$ is a subcategory of $\im{i_*}$.
        
        Since $j_!j^*(A)$ is bounded in cohomology, by the triangle
        \begin{equation*}
            \tri{j_!j^*(A)}{A}{i_*i^*(A)}{}{}{},
        \end{equation*}
        we have that $i_*i^*(A)$ is also bounded in cohomology. Thus, by Lemma~\ref{lem:homologyinsubcategories}, $i_*i^*(A)$ is in $\lsub{A}{\res{}{}{\quotient{A}{AeA}{}}}$. Moreover, $\function{i_*i^*}{\der{A}}{\der{A}}$ preserves set indexed coproducts and $i^*$ is essentially surjective so $\im{i_*} = \im{i_*i^*} = \lsub{A}{i_*i^*(A)}$. Thus $\im{i_*}$ is a subcategory of $\lsub{A}{\res{}{}{\quotient{A}{AeA}{}}}$.
        
        \item Similarly, to (i), $\cosub{A}{\res{}{}{\cogenerator{(\quotient{A}{AeA}{})}}}$ is a subcategory of $\im{i_*}$.
        
        Note that for all $m \in \integers$,
        \begin{equation*}
            \cohomology{m}{j_*j^*(\cogenerator{A})} \cong \dHom{A}{A}{j_*j^*(\cogenerator{A})[m]} 
            \cong \dHom{A}{j_!j^*(A)[-m]}{\cogenerator{A}}.
        \end{equation*}
        Thus, since $j_!j^*(A)$ is bounded in cohomology and $\cogenerator{A}$ is an injective cogenerator of $\Mod{A}$, $j_*j^*(\cogenerator{A})$ is bounded in cohomology. Hence $i_*i^!(\cogenerator{A})$ is also bounded in cohomology by the triangle
        \begin{equation*}
            \tri{i_*i^!(\cogenerator{A})}{\cogenerator{A}}{j_*j^*(\cogenerator{A})}{}{}{}.
        \end{equation*}
        Consequently, similarly to (i), $\im{i_*}$ is a subcategory of $\cosub{A}{\res{}{}{\cogenerator{(\quotient{A}{AeA}{})}}}$ using Lemma~\ref{lem:homologyinsubcategories}.
    \end{enumerate}
\end{proof}

\begin{prop} \label{prop:recollementmodcategoriesiga/aea+igeaeimpliesiga}
    Let $A$ be a ring and $e \in A$ an idempotent such that $\ltens{Ae}{eAe}{eA}$ is bounded in cohomology.
    \begin{enumerate}[label = (\roman*)]
        \item Suppose that $\quotient{A}{AeA}{}$ has finite flat dimension as a left $A$-module. If injectives generate for both $\quotient{A}{AeA}{}$ and $eAe$ then injectives generate for $A$.
        
        \item  Suppose that $\quotient{A}{AeA}{}$ has finite projective dimension as a right $A$-module. If projectives cogenerate for both $\quotient{A}{AeA}{}$ and $eAe$ then projectives cogenerate for $A$.
    \end{enumerate}
\end{prop}

\begin{proof}
    Recall that there exists a recollement of dg algebras, $(\der{B},\der{A},\der{eAe})$ with $B$ a dg algebra and $j^* = \tens{-}{A}{Ae}$.

    \begin{enumerate}[label = (\roman*)]
        \item Denote the restriction functor $\res{A}{\quotient{A}{AeA}{}}{}$ as $\res{}{}{}$ and the induction functor $\lind{A}{\quotient{A}{AeA}{}}{}$ as $\lind{}{}{}$. Then, by Lemma~\ref{lemma:imi*isgeneratedbyA/AeA},
        \begin{equation*}
            \im{i_*} = \kernel{j^*} = \lsub{A}{\res{}{}{\quotient{A}{AeA}{}}}.
        \end{equation*}
        Suppose that injectives generate for $\quotient{A}{AeA}{}$. Since $\res{}{}{}$ preserves set indexed coproducts $\im{\res{}{}{}}$ is a subcategory of $\lsub{A}{\res{}{}{\Inj{(\quotient{A}{AeA}{})}}}$ by Proposition~\ref{prop:imageoffunctor}. In particular, $\lsub{A}{\res{}{}{\quotient{A}{AeA}{}}}$ is a subcategory of $\lsub{A}{\res{}{}{\Inj{(\quotient{A}{AeA}{})}}}$. Moreover, if $\quotient{A}{AeA}{}$ has finite flat dimension as a left $A$-module then $\lind{}{}{}$ preserves complexes bounded in cohomology. Thus, by Lemma~\ref{lemma:adjointfunctorspreserve}, $\res{}{}{}$ preserves bounded complexes of injectives and $\res{}{}{\Inj{\left(\quotient{A}{AeA}{}\right)}}$ is a subcategory of $\lsubinj{A}$. Thus $\im{i_*}$ is a subcategory of $\lsubinj{A}$ and the result follows from Proposition~\ref{prop:recollementimi*}.
    
        \item By Lemma~\ref{lemma:imi*isgeneratedbyA/AeA},
        \begin{equation*}
            \im{i_*} = \kernel{j^*} = \cosub{A}{\res{}{}{\cogenerator{(\quotient{A}{AeA}{})}}}.
        \end{equation*}
        Moreover, if projectives cogenerate for $\quotient{A}{AeA}{}$ then, by Proposition~\ref{prop:imageoffunctor}, $\im{\res{}{}{}}$ is a subcategory of $\cosub{A}{\res{}{}{\Proj{(\quotient{A}{AeA}{})}}}$
        If $\quotient{A}{AeA}{}$ has finite projective dimension as a right $A$-module then $\rcoind{}{}{}$ preserves complexes bounded in cohomology and $\res{}{}{}$ preserves bounded complexes of projectives by Lemma~\ref{lemma:adjointfunctorspreserve}. Hence $\res{}{}{\Proj{(\quotient{A}{AeA}{})}}$ is a subcategory of $\cosubproj{A}$. Thus the result follows from Proposition~\ref{prop:recollementimi*}.
    \end{enumerate}
\end{proof}

\subsection{Vertex removal}

Recollements of module categories can be used to study vertex removal operations applied to quiver algebras. Let $A = \quotient{kQ}{I}{}$ be a quiver algebra and let $e \in A$ be an idempotent. Following the ideas of Fuller and Saor\'{\i}n \cite[Section 1]{Fuller1992} and Green, Psaroudakis and Solberg \cite[Section 5]{Green2018} we wish to consider the dependencies between the algebras $A$, $eAe$ and $\quotient{A}{AeA}{}$, with respect to injective generation when the semi-simple $A$-module $\quotient{(1-e)A}{\rad{(1-e)A}}{}$ has finite projective dimension or finite injective dimension. Note that if this semi-simple module has finite projective dimension as an $A$-module then all $\quotient{A}{AeA}{}$-modules have finite projective dimension as $A$-modules.

\begin{lemma} \label{lemma:simplefinitedimimplesmodulesfinitedim}
     Let $A$ be a finite dimensional algebra over a field and $e \in A$ be an idempotent. Let $S$ be the semi-simple $A$-module $\quotient{(1-e)A}{\rad{(1-e)A}}{}$. Let $N$ be an $A$-module that is annihilated by $e$.
     \begin{enumerate}[label = (\roman*)]
         \item If $S$ has finite injective dimension as an $A$-module then $N$ has finite injective dimension as an $A$-module.
         \item If $S$ has finite projective dimension as an $A$-module then $N$ has finite projective dimension as an $A$-module.
     \end{enumerate}
\end{lemma}

\begin{proof}
    Since $Ne$ is zero the radical series of $N$ contains only direct summands of set indexed coproducts of $S$. Hence if $S$ has finite injective (projective) dimension then $N$ has finite injective (respectively projective) dimension.
\end{proof}

This idea was generalised to arbitrary ring homomorphisms by Fuller and Saor\'{\i}n \cite{Fuller1992} and then Green, Psaroudakis and Solberg \cite[Section 3]{Green2018}. In particular, given a ring homomorphism $\function{\lambda}{A}{B}$ they consider the $A$-relative projective global dimension of $B$,
\begin{equation*}
    \relprojdim{A}{B} \coloneqq \supremum{\projdim{A}{\res{A}{B}{M_B}} : M_B \in \Mod{B} }.
\end{equation*}
Similarly, they also consider the $A$-relative injective global dimension of $B$,
\begin{equation*}
    \relinjdim{A}{B} \coloneqq \supremum{\injdim{A}{\res{A}{B}{M_B}} : M_B \in \Mod{B} }.
\end{equation*}

Note that Lemma~\ref{lemma:simplefinitedimimplesmodulesfinitedim} shows that, for a finite dimensional algebra, if the semi-simple $A$-module $S = \quotient{(1-e)A}{\rad{(1-e)A}}{}$ has finite projective dimension then $\relprojdim{A}{\quotient{A}{AeA}{}} < \infty$. Similarly, if $S$ has finite injective dimension then $\relinjdim{A}{\quotient{A}{AeA}{}} < \infty$. Using this we can apply the results of Proposition~\ref{prop:recollementimi*}, to the vertex removal operation.

\begin{prop}
    Let $A$ be a ring and $e \in A$ be an idempotent such that $\ltens{Ae}{eAe}{eA}$ is bounded in cohomology.
    \begin{enumerate}[label = (\roman*)]
        \item Suppose that $\relinjdim{A}{\quotient{A}{AeA}{}} < \infty$. If injectives generate for $eAe$ then injectives generate for $A$. 
        \item Suppose that $\relprojdim{A}{\quotient{A}{AeA}{}} < \infty$. If projectives cogenerate for $eAe$ then projectives cogenerate for $A$. 
    \end{enumerate}
\end{prop}

\begin{proof}
    \begin{enumerate}[label = (\roman*)]
        \item Since $\ltens{Ae}{eAe}{eA}$ is bounded in cohomology, 
        \begin{equation*}
            \im{i_*} = \lsub{A}{\res{}{}{\quotient{A}{AeA}{}}},
        \end{equation*}
        by Lemma~\ref{lemma:imi*isgeneratedbyA/AeA}.
        Furthermore, as $\relinjdim{A}{\quotient{A}{AeA}{}} < \infty$, $\res{}{}{\quotient{A}{AeA}{}}$ has finite injective dimension as an $A$-module. Hence $\res{}{}{\quotient{A}{AeA}{}}$ is in $\lsubinj{A}$ and $\im{i_*}$ is a subcategory of $\lsubinj{A}$. Thus the proof of Proposition~\ref{prop:recollementimi*} applies.
        
        \item Follows similarly to (i).
    \end{enumerate}
\end{proof}

Green, Psaroudakis and Solberg show that if $\relprojdim{A}{\quotient{A}{AeA}{}} \leq 1$ then $\function{\pi}{A}{\quotient{A}{AeA}{}}$ is a homological ring epimorphism \cite[Proposition 3.5 (iv)]{Green2018}. Note that $\pi$ is a homological ring epimorphism if and only if $\res{A}{\quotient{A}{AeA}{}}{}$ is a homological embedding \cite[Corollary 3.13]{Psaroudakis2014}. In this situation the abelian recollement lifts to a recollement of derived module categories of algebras, not dg algebras \cite{Cline1996}, \cite[Proposition 2.1]{Huegel2017-dgalgebra}, \cite[Subsection 1.7]{AngeleriHuegelKoenigLiu2011}.

\begin{lemma}
    Let $A$ be a ring and $e \in A$ be an idempotent.
    \begin{enumerate}[label = (\roman*)]
     \item Suppose that $\relinjdim{A}{\quotient{A}{AeA}{}} \leq 1$. 
        \begin{enumerate}[label = (\alph*)]
            \item Injectives generate for $A$ if and only if injectives generate for $eAe$.
        \end{enumerate}
        Moreover, suppose that $A$ is a finite dimensional algebra over a field.
        \begin{enumerate}[resume]
            \item  Projectives cogenerate for $A$ if and only if projectives cogenerate for $eAe$.
        \end{enumerate}      
        \item Suppose that $\relprojdim{A}{\quotient{A}{AeA}{}} \leq 1$. 
        \begin{enumerate}
            \item Projectives cogenerate for $A$ if and only if projectives cogenerate for $eAe$.
        \end{enumerate}
        Moreover, suppose that $A$ is a finite dimensional algebra over a field.
        \begin{enumerate}[resume]
            \item Injectives generate for $A$ if and only if injectives generate for $eAe$.
        \end{enumerate}
    \end{enumerate}
\end{lemma}

\begin{proof}
    If either $\relinjdim{A}{\quotient{A}{AeA}{}} \leq 1$ or $\relprojdim{A}{\quotient{A}{AeA}{}} \leq 1$, then the restriction functor $\res{A}{\quotient{A}{AeA}{}}{}$ is a homological embedding \cite[Proposition 3.5 (iv), Remark 5.9]{Green2018}. Thus $\function{\pi}{A}{\quotient{A}{AeA}{}}$ is a homological ring epimorphism \cite[Corollary 3.13]{Psaroudakis2014}. Hence the recollement of module categories $(\Mod{(\quotient{A}{AeA}{})},\Mod{A},\Mod{eAe})$ lifts to a recollement of derived module categories of the same rings, $(R) = (\der{\quotient{A}{AeA}{}},\der{A},\der{eAe})$ with $i_* = \tens{-}{\quotient{A}{AeA}{}}{\quotient{A}{AeA}{}}$ and $j^* = \tens{-}{A}{Ae}$, \cite{Cline1996}, \cite[Subsection 1.7]{AngeleriHuegelKoenigLiu2011}, \cite[Proposition 2.1]{Huegel2017-dgalgebra}.

    \begin{enumerate}[label = (\roman*)]
        \item Now suppose that $\relinjdim{A}{\quotient{A}{AeA}{}} \leq 1$. We claim that $\quotient{A}{AeA}{}$ has finite global dimension. Denote $\res{A}{\quotient{A}{AeA}{}}{}$ by $\res{}{}{}$ and $\rcoind{A}{\quotient{A}{AeA}{}}{}$ by $\rcoind{}{}{}$. Let $N$ be an $\quotient{A}{AeA}{}$-module. Then $\res{}{}{N}$ has finite injective dimension as an $A$-module so $\rcoind{}{}{} \circ \res{}{}{N}$ is quasi-isomorphic to a bounded complex of injectives. Since $(R)$ is a recollement of derived module categories $\res{}{}{}$ is fully faithful as a functor of derived categories. Thus $\rcoind{}{}{} \circ \res{}{}{N}$ is quasi-isomorphic to $N$ and $N$ has finite injective dimension as an $\quotient{A}{AeA}{}$-module. Consequently, $\quotient{A}{AeA}{}$ has finite global dimension so injectives generate for $\quotient{A}{AeA}{}$ and projectives cogenerate for $\quotient{A}{AeA}{}$.
        
        Since $i_* = \res{}{}{}$ preserves bounded complexes of injectives Proposition~\ref{prop:recollementbdedbelow} applies.
    
        \item 
        
        Similarly, if $\relprojdim{A}{\quotient{A}{AeA}{}} \leq 1$ then $\quotient{A}{AeA}{}$ has finite global dimension by considering $\lind{}{}{} \circ \res{}{}{N}$ for each $\quotient{A}{AeA}{}$-module $N$. Moreover, $i_* = \res{}{}{}$ preserves bounded complexes of projectives. Thus the statements follow from Proposition~\ref{prop:recollementsbdedabove}.

    \end{enumerate}
    
\end{proof}

\bibliographystyle{alpha}
\bibliography{ring_constructions_generation}

\medskip

\noindent 
\footnotesize \textsc{C. Cummings, School of Mathematics, University of Bristol, Bristol, BS8 1UG, UK}
\\
\noindent  \textit{E-mail address:} 
{\texttt{c.cummings@bristol.ac.uk}}

\end{document}